\documentclass{amsart}

\usepackage[foot]{amsaddr}

\usepackage[breaklinks,bookmarks=false]{hyperref}
\hypersetup{%
colorlinks,%
linkcolor=blue,%
citecolor=blue,%
urlcolor=blue,%
plainpages=false,%
pdfwindowui=false,%
pdfstartview={FitH}%
}

\usepackage{pgfplots}
\usepackage{xspace}
\usepackage{subcaption}

\usepackage{amssymb}
\usepackage{mathrsfs}

\newcommand{\eq}{:=}

\newcommand{\grad}{\boldsymbol \nabla}
\renewcommand{\div}{\grad \cdot}
\newcommand{\curl}{\grad \times}

\newcommand{\ccurl}{\boldsymbol{\operatorname{curl}}}
\newcommand{\ddiv}{\operatorname{div}}

\newcommand{\BA}{\boldsymbol A}

\newcommand{\BH}{\boldsymbol H}

\newcommand{\BJ}{\boldsymbol J}

\newcommand{\BL}{\boldsymbol L}

\newcommand{\BS}{\boldsymbol S}

\newcommand{\ba}{\boldsymbol a}
\newcommand{\bb}{\boldsymbol b}
\newcommand{\bc}{\boldsymbol c}

\newcommand{\bg}{\boldsymbol g}

\newcommand{\bl}{\boldsymbol l}

\newcommand{\bn}{\boldsymbol n}
\newcommand{\bo}{\boldsymbol o}
\newcommand{\bp}{\boldsymbol p}

\newcommand{\bu}{\boldsymbol u}
\newcommand{\bv}{\boldsymbol v}
\newcommand{\bw}{\boldsymbol w}
\newcommand{\bx}{\boldsymbol x}

\newcommand{\CE}{\mathcal E}
\newcommand{\CF}{\mathcal F}

\newcommand{\CH}{\mathcal H}

\newcommand{\CO}{\mathcal O}
\newcommand{\CP}{\mathcal P}

\newcommand{\CR}{\mathcal R}

\newcommand{\CT}{\mathcal T}

\newcommand{\CV}{\mathcal V}

\newcommand{\LU}{\mathscr U}

\newcommand{\BCH}{\boldsymbol{\CH}}

\newcommand{\BCP}{\boldsymbol{\CP}}

\newcommand{\BCR}{\boldsymbol{\CR}}

\newcommand{\RT}{\boldsymbol{RT}}
\newcommand{\ND}{\boldsymbol{N}}

\newcommand{\edge}{\ell}

\newcommand{\psie}{\boldsymbol \psi_\edge}
\newcommand{\taue}{\boldsymbol \tau_\edge}
\newcommand{\sige}{\boldsymbol \sigma_\edge}
\newcommand{\sigeh}{\boldsymbol \sigma_{\edge,h}}

\newcommand{\ome}{\omega_\edge}
\newcommand{\CTe}{\CT_h^\edge}

\newcommand{\he}{h_\edge}
\newcommand{\pe}{p_\edge}
\newcommand{\qe}{q_\edge}

\newcommand{\CPe}{C_{{\rm P},\edge}}
\newcommand{\Cconte}{C_{{\rm cont},\edge}}
\newcommand{\Cste}{C_{{\rm st},\edge}}
\newcommand{\osce}{\operatorname{osc}_\edge}

\newcommand{\oscK}{\operatorname{osc}_K}

\newcommand{\Clift}{C_{{\rm L},\Omega}}

\newcommand{\GD}{\Gamma_{\rm D}}
\newcommand{\GN}{\Gamma_{\rm N}}

\newcommand{\cank}{\bu^k}

\newcommand{\BLba}{\boldsymbol \Lambda}

\newcommand{\psia}{\psi_{\ba}}
\newcommand{\psib}{\psi_{\bb}}

\newcommand{\bzero}{\bo}

\newcommand{\supp}{\operatorname{supp}}

\newcommand{\mmg}{{\tt mmg3D}\xspace}
\newcommand{\gmsh}{{\tt gmsh}\xspace}
\newcommand{\hmax}{h_{\rm max}}

\newcommand{\err}{\operatorname{err}_\Omega}
\newcommand{\effe}{\operatorname{eff}_{\rm edge}}
\newcommand{\effc}{\operatorname{eff}_{\rm cell}}

\newcommand{\SlopeTriangle}[6]
{
    \pgfplotsextra
    {
        \pgfkeysgetvalue{/pgfplots/xmin}{\xmin}
        \pgfkeysgetvalue{/pgfplots/xmax}{\xmax}
        \pgfkeysgetvalue{/pgfplots/ymin}{\ymin}
        \pgfkeysgetvalue{/pgfplots/ymax}{\ymax}

        \pgfmathsetmacro{\xArel}{#1}
        \pgfmathsetmacro{\yArel}{#3}
        \pgfmathsetmacro{\xBrel}{#1-#2}
        \pgfmathsetmacro{\yBrel}{\yArel}
        \pgfmathsetmacro{\xCrel}{\xArel}
        \pgfmathsetmacro{\lnxB}{\xmin*(1-(#1-#2))+\xmax*(#1-#2)} % in [xmin,xmax].
        \pgfmathsetmacro{\lnxA}{\xmin*(1-#1)+\xmax*#1} % in [xmin,xmax].
        \pgfmathsetmacro{\lnyA}{\ymin*(1-#3)+\ymax*#3} % in [ymin,ymax].
        \pgfmathsetmacro{\lnyC}{\lnyA+#4*(\lnxA-\lnxB)}
        \pgfmathsetmacro{\yCrel}{\lnyC-\ymin)/(\ymax-\ymin)} % THE IMPROVED EXPRESSION WITHOUT 'DIMENSION TOO LARGE' ERROR.

        \coordinate (A) at (rel axis cs:\xArel,\yArel);
        \coordinate (B) at (rel axis cs:\xBrel,\yBrel);
        \coordinate (C) at (rel axis cs:\xCrel,\yCrel);

        \draw[#6]   (A)--
                    (B)--
                    (C)-- node[anchor=east] {#5}
                    cycle;
    }
}

\newtheorem{theorem}{Theorem}
\newtheorem{lemma}[theorem]{Lemma}
\newtheorem{proposition}[theorem]{Proposition}
\newtheorem{remark}[theorem]{Remark}

\numberwithin{theorem}{section}
\numberwithin{equation}{section}
\numberwithin{figure}{section}

\title[\MakeLowercase{$p$}-robust estimators for the curl-curl problem]%
{A simple equilibration procedure leading to polynomial-degree-robust
a posteriori error estimators for the curl-curl problem}

\author{T. Chaumont-Frelet$^{\star,\dagger}$}
\address{\vspace{-.5cm}}
\address{\noindent \tiny \textup{$^\star$Inria, 2004 Route des Lucioles, 06902 Valbonne, France}}
\address{\noindent \tiny \textup{$^\dagger$Laboratoire J.A. Dieudonn\'e, Parc Valrose, 28 Avenue Valrose, 06108 Nice, France}}

\begin{document}

\maketitle

\begin{abstract}
We introduce two a posteriori error estimators for N\'ed\'elec finite element discretizations
of the curl--curl problem. These estimators pertain to a new Prager--Synge identity and an
associated equilibration procedure. They are reliable and efficient, and the error estimates
are polynomial-degree-robust. In addition, when the domain is convex, the reliability constants
are fully computable. The proposed error estimators are also cheap and easy to implement,
as they are computed by solving divergence-constrained minimization problems over edge patches.
Numerical examples highlight our key findings, and show that both estimators are suited to drive
adaptive refinement algorithms. Besides, these examples seem to indicate that guaranteed upper
bounds can be achieved even in non-convex domains.

\vspace{.5cm}
\end{abstract}

\section{Introduction}

Given a domain $\Omega$, a partition $\{\GD,\GN\}$ of its boundary,
and a divergence free field $\BJ: \Omega \to \mathbb R^3$, the curl--curl
problem consists in finding $\BA: \Omega \to \mathbb R^3$ such that
\begin{subequations}
\label{eq_curlcurl_strong}
\begin{alignat}{2}
&\curl \curl \BA = \BJ, \quad \div \BA = 0,               &\qquad&\text{ in $\Omega$}, \\
&\BA \times \bn = \bzero,                                 &\qquad&\text{ on $\GD$},    \\
&(\curl \BA) \times \bn = \bzero, \quad \BA \cdot \bn = 0 &\qquad&\text{ on $\GN$},
\end{alignat}
\end{subequations}
together with a finite set of suitable orthogonality conditions to filter
out harmonic forms when the topology is non-trivial \cite{fernandes_gilardi_1997a}.
This problem is the central model for magnetostatic applications and is the basis
of Maxwell's equations, which are instrumental in the modeling of electromagnetic
phenomena \cite{griffiths_1999a}.

In general geometries, numerical schemes are required to approximate the solution
to \eqref{eq_curlcurl_strong}. Here, we consider finite element methods
\cite{ciarlet_2002a,monk_2003a,nedelec_1980a}, and focus on a posteriori error
estimation for N\'ed\'elec elements. This topic is already largely covered in the literature
\cite{%
beck_hiptmair_hoppe_wohlmuth_2000a,%
braess_schoberl_2008a,%
chaumontfrelet_ern_vohralik_2021a,%
chaumontfrelet_vohralik_2021a,%
gedicke_geevers_perugia_2019a,%
nicaise_creuse_2003a,%
schoberl_2008a},
motivated by the variety of important applications as well as the mathematical
challenges involved.

%%%

A posteriori error estimators for the curl--curl problem where first proposed
in \cite{beck_hiptmair_hoppe_wohlmuth_2000a} for convex domains
(see \cite[Assumption 2]{beck_hiptmair_hoppe_wohlmuth_2000a} for details).
The estimator in \cite{beck_hiptmair_hoppe_wohlmuth_2000a} is of residual type,
and has been generalized to arbitrary polyhedral Lipschitz domains in
\cite{nicaise_creuse_2003a,schoberl_2008a}. While this approach provides
reliable and efficient estimators, it still suffers from two drawbacks, namely:
(i) the constants appearing in the reliability estimates are not computable in practice,
and (ii) the constants in the efficiency estimates deteriorate as the polynomial degree
$p$ is increased. Notice that this is not specific to the curl--curl problem, and
residual-based estimators exhibit the same downsides even in the simpler context of
scalar elliptic problems \cite{ainsworth_oden_2000a,melenk_wohlmuth_2001a}.

In this work, we focus on so-called ``equilibrated'' error estimators,
that have the ability to provide (i) guaranteed upper bounds free of
unknown constants and (ii) polynomial-degree-robust (or simply, $p$-robust)
efficiency constants. The concept of equilibrated flux can be traced back to the seminal work
of Prager and Synge \cite{prager_synge_1947a}, where the authors establish that
equilibrated fluxes can be employed to provide guaranteed error upper bounds.
For scalar elliptic problems, several constructions of equilibrated fluxes
have then been proposed, leading to practically usable error estimators, see e.g.
\cite{ainsworth_oden_2000a,ladeveze_leguillon_1983a,luce_wohlmuth_2004a}.
Here, we focus on the approach initially proposed in \cite{destuynder_metivet_1999a}
and later extended in \cite{ern_vohralik_2015a}. It relies on a partition of unity
via finite element shape functions and local mixed finite element problems, and lead to
$p$-robust estimates \cite{braess_pillwein_schoberl_2009a,ern_vohralik_2018a}.

For the curl--curl problem, the construction of equilibrated fluxes turns out to
be a much more arduous task than for scalar elliptic equations. A procedure for
the lowest-order N\'ed\'elec elements have been introduced early \cite{braess_schoberl_2008a},
but its generalization to arbitrary orders has only been proposed recently
\cite{gedicke_geevers_perugia_2019a}. Actually, to the best of the author's knowledge,
they are currently only two constructions of equilibrated estimator that lead to $p$-robust
estimates for the curl--curl problem, namely (i) the construction in
\cite{gedicke_geevers_perugia_schoberl_2020a} which is based on
\cite{gedicke_geevers_perugia_2019a}, and (ii) the approach in
\cite{chaumontfrelet_vohralik_2021a} which employs a partition of unity in the spirit of
\cite{destuynder_metivet_1999a,ern_vohralik_2015a}. While the
estimators in \cite{chaumontfrelet_vohralik_2021a} and
\cite{gedicke_geevers_perugia_schoberl_2020a} provide constant-free reliability estimates
and $p$-robust efficiency constants, there are not fully satisfactory as there construction is
complicated. Indeed, they hinge on over-constrained minimization problems, and require
several passes through the mesh. In addition, the patches involved in the efficiency estimates
are rather large.

Another approach called ``broken patchwise equilibration'' has been proposed in
\cite{chaumontfrelet_ern_vohralik_2021a}, where the authors introduce a $p$-robust
a posteriori error estimator for the curl--curl problem that do not rely on equilibration.
The key assets of this estimator is that it is cheap and straightforward to compute
(especially as compared to
\cite{chaumontfrelet_vohralik_2021a,gedicke_geevers_perugia_schoberl_2020a}),
and that the efficiency estimates are established on tight edge patches. On the other hand,
because the reliability estimate does not pertain to a Prager--Synge identity, the reliability
bound contains constants that are either unavailable, or cumbersome to compute in practice.

The equilibrated estimators of 
\cite{%
braess_schoberl_2008a,%
chaumontfrelet_vohralik_2021a,%
gedicke_geevers_perugia_2019a,%
gedicke_geevers_perugia_schoberl_2020a}
are based on the following Prager--Synge identity: if
$\BH \in \BH_{\GN}(\ccurl,\Omega)$ is any ``equilibrated field'' satisfying $\curl \BH = \BJ$,
then
\begin{equation}
\label{eq_prager_synge_curl}
(\curl(\BA-\BA_h),\curl \bv)_\Omega = (\BH-\curl\BA_h,\curl \bv)_\Omega,
\end{equation}
for any field $\BA_h,\bv \in \BH_{\GD}(\ccurl,\Omega)$
(the notations are rigorously introduced in Section \ref{section_setting} below),
leading to the estimate
\begin{equation}
\label{eq_prager_synge_curl_estimate}
\|\curl(\BA-\BA_h)\|_\Omega \leq \|\BH-\curl\BA_h\|_\Omega.
\end{equation}
This motivates the construction of equilibrated fields $\BH$ with prescribed curl.
Unfortunately, as previously mentioned, the construction of such fields is rather
involved
\cite{%
chaumontfrelet_vohralik_2021a,%
gedicke_geevers_perugia_2019a,%
gedicke_geevers_perugia_schoberl_2020a}.

Here, we introduce an alternative Prager--Synge identity that leads to
a much simpler equilibration procedure. It relies on the observation that whenever
$v \in H^1(\Omega)$, we have $\curl (v\cank) = \grad v \times \cank$, where $\{\cank\}_{k=1}^3$
is the canonical basis of $\mathbb R^3$. Simple manipulations then show the following identity:
if, for each $k \in \{1,2,3\}$, $\BS^k \in \BH_{\GN}(\ddiv,\Omega)$ satisfies
$\div \BS^k = \BJ_k$, we have
\begin{equation}
\label{eq_prager_synge_grad}
(\curl(\BA-\BA_h),\curl \bv)_\Omega
=
-\sum_{k=1}^3 (\BS^k+\cank \times \curl \BA_h,\grad \bv_k)_\Omega
\end{equation}
for all $\BA_h \in \BH_{\GD}(\ccurl,\Omega)$ and $\bv \in \BH^1_{\GD}(\Omega)$.
Interestingly, as opposed to \eqref{eq_prager_synge_curl}, identity \eqref{eq_prager_synge_grad}
only constrains the divergence of the equilibrated fields, simplifying the
equilibration process. On the other hand, the downside of \eqref{eq_prager_synge_grad}
is that it requires the gradient of the test function $\bv \in \BH^1_{\GD}(\Omega)$ instead
of its curl in \eqref{eq_prager_synge_curl}. As usual \cite{nicaise_creuse_2003a,schoberl_2008a},
this is remedied by employing ``regular decompositions'', leading to the estimate
\begin{equation}
\label{eq_prager_synge_grad_estimate}
\|\curl(\BA-\BA_h)\|_\Omega
\leq
\Clift
\left (
\sum_{k=1}^3 \|\BS^k+\cank \times \curl \BA_h\|_\Omega,
\right )^{1/2},
\end{equation}
where the constant $\Clift$ is one in convex domains (with $\GD = \partial \Omega$
or $\GN = \partial \Omega$), but is not practically computable in general
(see Section \ref{section_regular_decomposition}).

% While this may appear as
% a large drawback, numerical examples indicate that the presence of the $\Clift$
% may be spurious, leading to very interesting error estimators in practice.
% We also elaborate on the necessity of the constant $\Clift$ in Remark
% \ref{remark_clift}.

In this work, we carefully establish the new Prager--Synge identity
introduced in \eqref{eq_prager_synge_grad}, and we use it to design
and analyze two novel a posteriori error estimators for curl--curl problem
\eqref{eq_curlcurl_strong}. We show that they are reliable and efficient
with $p$-robust constants. We also provide a set of numerical experiments
showing that both estimators are suitable to drive adative mesh refinements
processes.
Our estimators share some similarities with the broken patchwise estimator
introduced in \cite{chaumontfrelet_ern_vohralik_2021a} and the equilibrated
estimators based on \eqref{eq_prager_synge_curl} provided in
\cite{chaumontfrelet_vohralik_2021a,gedicke_geevers_perugia_schoberl_2020a},
which we now describe.

As in \cite{chaumontfrelet_ern_vohralik_2021a}, our construction relies on a
vectorial partition of unity via edge functions (see Section \ref{section_partition_of_unity}),
which leads to a localization on tight edge patches. In contrast to
\cite{chaumontfrelet_ern_vohralik_2021a} however, the proposed approach has two
main assets. Indeed, (i) it uses divergence-constrained minimization problems which
are easier to implement and cheaper to solve than the curl-constrained problems employed
in \cite{chaumontfrelet_ern_vohralik_2021a}. Besides, (ii) the present approach hinges
on an equilibration principle, which leads to nicer upper bounds than in
\cite{chaumontfrelet_ern_vohralik_2021a}. The two approaches can be seen as dual to one
another in some sense, as we elaborate in Remark \ref{remark_broken_estimator}.

% . However, the estimators we present are based on
% divergence-constrained minimization problems, instead of curl-constrained problems,
% which ease the implementation, and decreases the computational cost. As compated to
% \cite{chaumontfrelet_ern_vohralik_2021a}, the present estimators also have the advantage
% to hinge on an equilibration principle given by \eqref{eq_prager_synge_grad}, which leads to
% nicer reliability estimates.

As compared to the usual equilibration procedures of
\cite{%
braess_schoberl_2008a,
chaumontfrelet_vohralik_2021a,%
gedicke_geevers_perugia_2019a,%
gedicke_geevers_perugia_schoberl_2020a}
based on \eqref{eq_prager_synge_curl}, the main drawback of the present
approach is that apart from convex domains, the upper bound contains
the unknown constant $\Clift$. While this appears as a major downside, numerical
experiments suggest that the constant $\Clift$ might be spurious, and that
taking $\Clift = 1$ provides guaranteed upper bounds, even in non-convex domains.
We also elaborate on why the constant $\Clift$ may not be necessary in Remark
\ref{remark_clift}.

% The key advantages of these two estimators are the following. (i) They only require one pass
% through the mesh, leading to cheaper and easier implementations and tighter mesh patches in the
% efficiency estimates, as compared to the original Prager--Synge identity
% \cite{chaumontfrelet_vohralik_2021a,gedicke_geevers_perugia_2019a}. (ii) Similar
% to \cite{chaumontfrelet_ern_vohralik_2021a}, the construction only involves problems
% set on edge patches (smaller than the usual vertex patches), leading to mixed-finite
% element problems with only a handful of degrees of freedom. (iii) As opposed to
% \cite{%
% chaumontfrelet_ern_vohralik_2021a,%
% chaumontfrelet_vohralik_2021a,%
% gedicke_geevers_perugia_2019a},
% the estimators are constructed using divergence-constrained minimization problems
% (instead of curl-constrained problems), further decreasing the computational cost.

The remaining of this work is organized as follows. In Section \ref{section_setting},
we describe the setting, and recall preliminary results. Section \ref{section_localization}
presents the key abstract arguments to localize the construction of the estimators.
We introduce our two estimators in Sections \ref{section_edge_estimator}
and \ref{section_equilibrated_estimator} where we establish that they are reliable and
efficient with $p$-robust constants. Section \ref{section_numerical_examples} illustrates
the key theoretical findings with a set of numerical examples, and we present concluding
remarks in Section \ref{section_conclusion}.

\section{Setting}
\label{section_setting}

\subsection{Domain}

We consider a Lipschitz polyhedral domain $\Omega \subset \mathbb R^3$.
The boundary of $\Omega$ is split into two disjoint (relatively open) parts
$\GD$ and $\GN$ in such a way that $\partial \Omega = \overline{\GD} \cup \overline{\GN}$.
We assume that both $\GD$ and $\GN$ are polygonal. Notice that we do not
assume that $\Omega$ is simply connected, nor that $\GD$ or $\GN$ are connected.

\subsection{Functional spaces}

If $\omega \subset \Omega$ is an open set, $L^2(\omega)$ and $\BL^2(\omega)$
are the spaces of scalar and vector-valued square integrable functions defined
on $\omega$. The usual inner products and norms of these spaces are denoted
by $(\cdot,\cdot)_\omega$ and $\|\cdot\|_\omega$. $H^1(\omega)$, $\BH(\ddiv,\omega)$
and $\BH(\ccurl,\omega)$ are the usual Sobolev spaces respectively containing
square-integrable functions with square-integrable gradient, divergence, and curl.
We also employ the notation $\BH^1(\omega) \eq \left (H^1(\omega)\right )^3$.
If $\Gamma \subset \partial \omega$ is a relatively open subset, $H^1_\Gamma(\omega)$
is the set of functions of $H^1(\omega)$ with vanishing trace on $\Gamma$,
and we also note $\BH^1_\Gamma(\omega) \eq \left (H^1_\Gamma(\omega)\right )^3$.
Similarly, we set
\begin{equation*}
\BH_\Gamma(\ddiv,\omega)
\eq
\left \{
\bv \in \BH(\ddiv,\omega); \; (\bv,\grad q) + (\div \bv,q) = 0 \quad
\forall q \in H^1_{\Gamma^{\rm c}}(\omega)
\right \}
\end{equation*}
and
\begin{equation*}
\BH_\Gamma(\ccurl,\omega)
\eq
\left \{
\bw \in \BH(\ddiv,\omega); \; (\bw,\curl \bp) - (\curl \bv,\bp) = 0 \quad
\forall \bp \in \BH^1_{\Gamma^{\rm c}}(\omega)
\right \}
\end{equation*}
with $\Gamma^{\rm c} \eq \partial \omega \setminus \overline{\Gamma}$.
We refer the reader to \cite{adams_fournier_2003a,fernandes_gilardi_1997a,girault_raviart_1986a}
for an in-depth presentation of the above spaces.

For $m \geq 2$, $H^m(\omega)$ is the space of functions $v \in L^2(\omega)$ such that
$\partial^{\boldsymbol \alpha} v \in L^2(\omega)$ for all $|\boldsymbol \alpha| \leq m$.
If $\LU$ is a collection of open sets $\omega$, then $H^m(\LU)$ contains those functions
$v \in L^2 (\bigcup \LU)$ such that $v|_\omega \in H^m(\omega)$ for all $\omega \in \LU$.
We also employ the notations $\BH^m(\omega) \eq \left (H^m(\omega)\right )^3$
and $\BH^m(\LU) \eq \left (H^m(\LU)\right )^3$. The usual seminorm of $H^m(\LU)$ is
denoted by
\begin{equation*}
|v|_{H^m(\LU)}^2
\eq
\sum_{\omega \in \LU}
\sum_{|\boldsymbol \alpha| = m} \|\partial^{\boldsymbol \alpha} v\|_\omega^2
\quad \forall v \in H^m(\omega),
\end{equation*}
and we set
\begin{equation*}
|\bv|_{\BH^m(\LU)}^2 \eq \sum_{k=1}^3 |\bv_k|_{H^m(\LU)}^2
\end{equation*}
when $\bv \in \BH^m(\LU)$.

We also employ the notation
\begin{equation*}
\BLba_{\GD}(\Omega) \eq \left \{
\bv \in \BH_{\GD}(\Omega); \; \curl \bv = \bzero
\right \},
\end{equation*}
for functions with vanishing curl. When $\Omega$ is simply connected and $\GD$ is
connected, we simply have $\BLba_{\GD}(\Omega) = \grad \left (H^1_{\GD}(\Omega)\right )$.
In the general case however,
$\BLba_{\GD}(\Omega) = \grad \left (H^1_{\GD}(\Omega)\right ) + \BCH_{\GD}(\Omega)$
where $\BCH_{\GD}(\Omega)$ is a finite dimensional ``cohomology'' space. The dimension
of $\BCH_{\GD}(\Omega)$ depends on the topology of $\Omega$ and $\GD$, and its structure
is explicitly known \cite{fernandes_gilardi_1997a}. We write $\BLba^\perp_{\GD}(\Omega)$
for the orthogonal complement of $\BLba_{\GD}(\Omega)$ in $\BL^2(\Omega)$.

\subsection{Model problem}

Assuming that $\BJ \in \BL^2(\Omega)$, the weak form of \eqref{eq_curlcurl_strong}
consists in finding $(\BA,\bl) \in \BH_{\GD}(\ccurl,\Omega) \times \BLba_{\GD}(\Omega)$
such that
\begin{equation}
\label{eq_curlcurl_weak}
\left \{
\begin{array}{rcl}
(\curl \BA,\curl \bv)_\Omega + (\bl,\bv)_\Omega &=& (\BJ,\bv)_\Omega
\\
(\BA,\bw)_\Omega &=& 0
\end{array}
\right .
\end{equation}
for all $(\bv,\bw) \in \BH_{\GD}(\ccurl,\Omega) \times \BLba_{\GD}(\Omega)$.
Problem \eqref{eq_curlcurl_weak} is a standard saddle-point problem, and
admits a unique solution \cite{brezzi_1974a}. We will assume throughout this work that
$\BJ \in \BLba_{\GD}^\perp(\Omega)$, so that actually, $\bl = \bzero$, and
\begin{equation}
(\curl \BA,\curl \bv)_\Omega = (\BJ,\bv)_\Omega
\quad
\forall \bv \in \BH_{\GD}(\ccurl,\Omega).
\end{equation}

\subsection{Computational mesh and edge patches}

The domain $\Omega$ is partitioned into a mesh $\CT_h$ of (open) tetrahedral elements $K$.
We assume that the mesh is conforming in the sense of \cite{ciarlet_2002a}, meaning that the
intersection $\overline{K_-} \cap \overline{K_+}$ of two distincts elements $K_\pm \in \CT_h$
is either empty, or it is a full face, edge or vertex of both $K_-$ and $K_+$.
The set vertices, edges and faces of the mesh are respectively denoted by $\CV_h$, $\CE_h$
and $\CF_h$. Classically, we assume that the mesh is conforming with the partition
of the boundary, i.e., that for all faces $F \in \CF_h$ such that $F \subset \partial \Omega$,
either $F \subset \GD$ or $F \subset \GN$. For each $K \in \CT_h$, the quantity
$\kappa_K \eq h_K/\rho_K$ (where, as usual, $h_K$ is the diameter of $K$ and $\rho_K$ is the
diameter of the largest ball contained in $\overline{K}$) is the shape-regularity
parameter of the element $K$.

Consider an edge $\edge \in \CE_h$. The mesh patch $\CTe$ gathers those elements
$K \in \CT_h$ having $\edge$ as an edge. Besides, we denote by $\ome$ the open domain
covering the elements $K \in \CTe$, and $\he$ is the diameter of $\ome$.
We also associate with $\edge$ an arbitrary, but fixed, unit tangent vector
$\taue \eq (\bb-\ba)/|\bb-\ba|$, where $\ba,\bb \in \CV_h$ are the two vertices
of $\edge$. We then introduce the edge function
\begin{equation}
\psie \eq |\bb-\ba| (\psia \grad \psib - \psib \grad \psia)
\end{equation}
where $\psia,\psib \in \CP_1(\CT_h) \cap H^1(\Omega)$ are the usual hat
functions satisfying $\psia(\bc) = \delta_{\ba,\bc}$ and
$\psib(\bc) = \delta_{\bb,\bc}$ for all $\bc \in \CV_h$. Notice that
$\psie|_{\edge'} \cdot \boldsymbol \tau_{\edge'} = \delta_{\edge,\edge'}$
for all $\edge' \in \CE_h$. We also denote by $\kappa_\edge \eq \min_{K \in \CTe} \kappa_K$
the shape-regularity paremeter of the edge patch.

\subsection{Finite element spaces}

For all integer $q \geq 0$ and elements $K \in \CT_h$, $\CP_q(K)$
stands for the space of polynomials from $K$ to $\mathbb R$ of degree less than or equal to
$q$, and $\BCP_q(K) \eq \left (\CP_q(K)\right )^3$. As usual
\cite{nedelec_1980a,raviart_thomas_1977a}, the Raviart-Thomas and N\'ed\'elec polynomial
spaces are respectively defined by
\begin{equation*}
\RT_q(K) \eq \bx \CP_q(K) + \BCP_q(K),
\quad
\ND_q(K) \eq \bx \times \BCP_q(K) + \BCP_q(K).
\end{equation*}

If $\CT \subset \CT_h$ and $\omega$ is the domain covering the
elements of $\CT$, then $\CP_q(\CT)$ gathers the functions
$u: L^2(\omega) \to \mathbb R$ such that $u|_K \in \CP_q(K)$
for all $K \in \CT_h$. Notice that there are no ``build in''
compatibility conditions in this space. We also define
$\BCP_q(\CT)$, $\RT_q(\CT)$ and $\ND_q(\CT)$ in a similar fashion.

If $\bv \in \BL^2(\omega)$, we denote by $\pi_{\CT,q}(\bv) \in \BCP_q(\CT)$
its $\BL^2(\omega)$-projection over $\BCP_q(\CT)$, which is uniquely defined by
\begin{equation*}
(\bv-\pi_{\CT,q}(\bv),\bw_h)_\omega = 0 \qquad \forall \bw_h \in \BCP_q(\CT),
\end{equation*}
and we have the estimate
\begin{equation}
\label{eq_poincare_Kq}
\|\bv-\pi_{\CT,q}(\bv)\|_\omega
\leq
\left (
\frac{h_\CT}{\pi}
\right )^{q+1}
|\bv|_{\BH^{q+1}(\CT)}
\end{equation}
with $h_\CT \eq \max_{K \in \CT} h_K$, whenever $\bv \in \BH^{q+1}(\CT)$.
We also employ the symbol $\pi_{\CT,q}$ for the $L^2(\omega)$-projection
over $\CP_q(\CT)$. Notice that since $K \in \CT_h$ is convex, \eqref{eq_poincare_Kq}
can be obtained by repeated applications of the Poincar\'e inequality established
in \cite{payne_weinberger_1960a} to $\bv_k$, $1 \leq k \leq 3$, and its derivatives
$\partial^{\boldsymbol \alpha} \bv_k$ for $|\boldsymbol \alpha| \leq q$.

\subsection{Discrete solution}
\label{section_discrete_solution}

Throughout this work, we consider a discrete function
$\BA_h \in \BH_{\GD}(\ccurl,\Omega)$ such that
\begin{equation}
\label{eq_polynomial_degree}
\BA_h|_{\ome} \in \ND_{\pe}(\CTe)
\end{equation}
for some $\pe \geq 0$ for $\edge \in \CE_h$ and
\begin{equation}
\label{eq_galerkin_orthogonality}
(\curl \BA_h,\curl \psie)_\Omega = (\BJ,\psie)_\Omega
\end{equation}
for all edge $\edge \in \CE_h$ such that $\edge \not \subset \overline{\GD}$.
Observe that in particular, \eqref{eq_galerkin_orthogonality} typically holds true if $\BA_h$
is obtained from a Galerkin approximation to \eqref{eq_curlcurl_weak}, since the edge
functions involved in \eqref{eq_galerkin_orthogonality} belong to the lowest-order N\'ed\'elec
space. On the other hand, \eqref{eq_polynomial_degree} is satisfied if $\BA_h$ belongs
to a $hp$-adaptive N\'ed\'elec finite element space, and we can simply take $\pe = p$ for all
$\edge \in \CE_h$ in the particular case of a uniform polynomial degree $p \geq 0$.

% \section{Preliminary results}

\subsection{Regular decomposition}
\label{section_regular_decomposition}

For all $\bv \in \BH_{\GD}(\ccurl,\Omega)$, there exists a function $\bw \in \BH^1_{\GD}(\Omega)$
such that $\curl \bw = \curl \bv$ and
\begin{equation}
\label{eq_regular_decomposition}
\|\grad \bw\|_\Omega \leq \Clift \|\curl \bv\|_\Omega.
\end{equation}
In addition, if $\Omega$ is convex and either $\GD$ or $\GN$ is empty, we can
take $\Clift = 1$. We refer the reader to
\cite{costabel_dauge_nicaise_1999a,girault_raviart_1986a,hiptmair_pechstein_2017a}
for a proof of these results.

\subsection{Partition of unity via edge functions}
\label{section_partition_of_unity}

For all $\edge \in \CE_h$, $\supp \psie = \overline{\ome}$, and we have
\begin{equation}
\label{eq_scaling_edge_functions}
\|\psie\|_{L^\infty(\ome)} + \he \|\curl \psie\|_{L^\infty(\ome)} \leq C(\kappa_\edge)
\quad
\forall \edge \in \CE_h,
\end{equation}
where $C(\kappa_\edge)$ is a constant that only depends on the shape regularity parameter
$\kappa_\edge$ of the edge patch. In addition, the identity
\begin{equation}
\label{eq_partition_of_unity}
\bw = \sum_{\edge \in \CE_h} (\bw \cdot \taue) \psie
\end{equation}
holds true for all $\bw \in \BL^2(\Omega)$. We refer the reader
to \cite[Section 5.3]{chaumontfrelet_ern_vohralik_2021a} for a proof of these facts.

\subsection{Local functional spaces and Poincar\'e inequalities}
\label{section_local_spaces}

Consider an edge $\edge \in \CE_h$. If $\edge \not \subset \overline{\GD}$,
$H^1_\star(\ome)$ is the subset of $H^1(\ome)$ of functions with vanishing
mean value and $\BH_0(\ddiv,\ome) \eq \BH_{\partial \ome}(\ddiv,\ome)$.
If on the other hand, $\edge \subset \overline{\GD}$,
$H^1_\star(\ome) \eq H^1_{\Gamma}(\ome)$ and
$\BH_0(\ddiv,\ome) \eq \BH_{\Gamma^{\rm c}}(\ddiv,\ome)$
where $\overline{\Gamma} \eq \partial \ome \cap \GD$ and
$\Gamma^{\rm c} \eq \partial \omega \setminus \GD$.

For all edges $\edge \in \CE_h$, the constant
\begin{equation}
\label{eq_poincare_ome}
\CPe \eq \he^{-1} \sup_{\substack{w \in H^1_\star(\ome) \\ \|\grad w\|_{\ome} = 1}} \|w\|_{\ome}
\end{equation}
is finite and only depends on $\kappa_\edge$.

Since every cell $K \in \CT_h$ is convex \cite{payne_weinberger_1960a}, the following elementwise
Poincar\'e inequality holds true:
\begin{equation}
\label{eq_poincare_K}
\|\bv-\pi_{K,0} \bv\|_K \leq \frac{h_K}{\pi} \|\grad \bv\|_K \qquad \forall \bv \in \BH^1(K),
\end{equation}
where $\pi_{K,0}$ denotes the $\BL^2(K)$-projection of $\bv$ over $\BCP_0(K)$.

\section{Localization of the residual functional}
\label{section_localization}

In this section, we start by presenting abstract arguments
later used to localize the computation of the estimator.
Specifically, we introduce a residual functional
$\BCR \in \left (\BH_{\GD}(\ccurl,\Omega)\right )'$ by setting
\begin{equation}
\langle \BCR, \bv \rangle \eq (\BJ,\bv)_\Omega - (\curl \BA_h,\curl \bv)_\Omega
\end{equation}
for all $\bv \in \BH_{\GD}(\ccurl,\Omega)$. Observe that because of
\eqref{eq_galerkin_orthogonality}, we have
\begin{equation}
\label{eq_residual_orthogonality}
\langle \BCR,\psie \rangle = 0 \qquad \forall \edge \in \CE_h; \; \edge \not \subset \overline{\GD}.
\end{equation}
We will employ two key norms for the residual functional. On the one hand the usual dual norm
reads
\begin{equation}
\label{eq_definition_BCR_star_Omega}
\|\BCR\|_{\star,\Omega}
\eq
\sup_{\substack{\bv \in \BH_{\GD}(\ccurl,\Omega) \\ \|\curl \bv\| = 1}}
\langle \BCR,\bv \rangle.
\end{equation}
On the other hand, we also introduce a localized norm
\begin{equation}
\label{eq_definition_BCR_star_edge}
\|\BCR\|_{\star,\edge}
\eq
\sup_{\substack{w \in H^1_\star(\ome) \\ \|\grad w\|_{\ome} = 1}}
\langle \BCR, w \psie \rangle
\end{equation}
for each $\edge \in \CE_h$.

Notice that as $\langle \BCR,\bv \rangle = (\curl(\BA-\BA_h),\curl \bv)_\Omega$
for all $\bv \in \BH_{\GD}(\Omega)$, we have
\begin{equation}
\label{eq_dual_norm_error}
\|\BCR\|_{\star,\Omega} = \|\curl(\BA-\BA_h)\|_\Omega,
\end{equation}
so that the dual norm of $\BCR$ is the quantity we actually want to estimate.
Directly estimating $\|\BCR\|_{\star,\Omega}$, however, would lead to global (and
expensive) computations. Henceforth, the goal of this section is to establish
equivalence results between the dual norm $\|\BCR\|_{\star,\Omega}$ and the
squared sum of the localized norms $\|\BCR\|_{\star,\edge}$, which are more suitable
to tackle numerically.

We start by establishing an upper bound.

\begin{theorem}[Continous reliability]
\label{theorem_continuous_reliability}
We have
\begin{equation}
\label{eq_continuous_reliability}
\|\BCR\|_{\star,\Omega}
\leq
\sqrt{6} \Clift \left (
\sum_{\ell \in \CE_h} \|\BCR\|_{\star,\edge}^2
\right )^{1/2}.
\end{equation}
\end{theorem}

\begin{proof}
Let $\bv \in \BH_{\GD}(\ccurl,\Omega)$ with $\|\curl \bv\|_\Omega = 1$.
Recalling the discussion in Section \ref{section_regular_decomposition},
there exists $\bw \in \BH^1_{\GD}(\Omega)$ with $\curl \bv = \curl \bw$ and
\begin{equation}
\label{tmp_grad_bw}
\|\grad \bw\|_{\Omega} \leq \Clift \|\curl \bv\|_{\Omega} = \Clift.
\end{equation}
Then, we use \eqref{eq_partition_of_unity} and \eqref{eq_residual_orthogonality}
to show that
\begin{equation*}
\langle \BCR,\bv \rangle
=
\langle \BCR,\bw \rangle
=
\sum_{\edge \in \CE_h} \langle \BCR,(\bw\cdot\taue)\psie \rangle
=
\sum_{\edge \in \CE_h} \langle \BCR,(\bw\cdot\taue-\vartheta_\edge)\psie \rangle
\end{equation*}
with
\begin{equation*}
\vartheta_\edge \eq \left \{
\begin{array}{ll}
0 & \text{ if } \edge \subset \overline{\GD},
\\
\frac{1}{|\ome|}\int_{\ome} \bw \cdot \taue & \text{otherwise.}
\end{array}
\right .
\end{equation*}
We then observe that $\bw \cdot \taue - \vartheta_\edge \in H^1_\star(\ome)$,
so that recalling the definition of $\|\BCR\|_{\star,\edge}$ in
\eqref{eq_definition_BCR_star_edge}, we have
\begin{align*}
\langle \BCR,\bv \rangle
=
\sum_{\edge \in \CE_h} \|\BCR\|_{\star,\edge}\|\grad(\bw \cdot \taue)\|_{\ome}
&\leq
\sum_{\edge \in \CE_h} \|\BCR\|_{\star,\edge}\|\grad \bw\|_{\ome}
\\
&\leq
\sqrt{6}
\left (
\sum_{\edge \in \CE_h}
\|\BCR\|_{\star,\edge}^2
\right )^{1/2}
\|\grad \bw\|_\Omega,
\end{align*}
and \eqref{eq_continuous_reliability} follows from \eqref{tmp_grad_bw}.
\end{proof}

Next, we provide a lower bound.

\begin{theorem}[Continuous efficiency]
For all edges $\edge \in \CE_h$, we have
\begin{equation}
\label{eq_continuous_efficiency}
\|\BCR\|_{\star,\edge}
\leq
\Cconte \|\curl(\BA-\BA_h)\|_{\ome}
\end{equation}
where
\begin{equation*}
\Cconte
\eq
\|\psie\|_{L^\infty(\ome)} + \CPe \he\|\curl \psie\|_{L^\infty(\ome)}
\end{equation*}
only depends on $\kappa_\edge$.
\end{theorem}

\begin{proof}
First, we observe that there exists $w_\star \in H^1_\star(\ome)$
with $\|\grad w_\star\|_{\ome} = 1$ such that
\begin{align*}
\|\BCR\|_{\star,\edge}
=
\langle \BCR,w_\star\psie\rangle
&=
(\curl(\BA-\BA_h),\curl(w_\star\psie))_{\Omega}
\\
&\leq
\|\curl(\BA-\BA_h)\|_{\ome}\|\curl(w_\star\psie)\|_{\ome}.
\end{align*}
Then, \eqref{eq_continuous_efficiency} follows since
\begin{align*}
\|\curl(w_\star\psie)\|_{\ome}
&\leq
\|\curl \psie\|_{L^\infty(\ome)} \|w_\star\|_{\ome}
+
\|\psie\|_{L^\infty(\ome)} \|\grad w_\star\|_{\ome} \\
&\leq
\|\psie\|_{L^\infty(\ome)} + \CPe \he\|\curl \psie\|_{L^\infty(\ome)}.
\end{align*}
The fact that $\Cconte$ only depends on $\kappa_\edge$ is a direct
consequence of \eqref{eq_scaling_edge_functions} and \eqref{eq_poincare_ome}.
\end{proof}

Finally, we provide an alternative expression for $\|\BCR\|_{\edge,\star}$
by employing a duality argument. This dual expression corresponds to a minimization
problem and is the basis of the estimators we propose in this work.

\begin{theorem}[Dual characterization]
\label{theorem_dual_characterization}
For all $\edge \in \CE_h$, the equality
\begin{equation}
\label{eq_dual_characterization}
\|\BCR\|_{\edge,\star}
=
\|\sige + \psie \times \curl \BA_h\|_{\ome}
\end{equation}
holds true with
\begin{equation}
\label{eq_definition_sige}
\sige
\eq
\arg \min_{\substack{
\bv \in \BH_0(\ddiv,\ome) \\ \div \bv = \psie \cdot \BJ - \curl \psie \cdot \curl \BA_h
}}
\|\bv + \psie \times \curl \BA_h\|_{\ome}.
\end{equation}
\end{theorem}

\begin{proof}
Fix $\edge \in \CE_h$. We introduce a Riesz representant $r_\edge$ defined as the
unique element of $H^1_\star(\ome)$ such that
\begin{equation}
\label{tmp_definition_r_edge}
(\grad r_\edge,\grad w)_{\ome}
=
\langle \BCR, w\psie \rangle
\end{equation}
for all $w \in H^1_\star(\ome)$. Since
\begin{align*}
\langle \BCR, w\psie \rangle
&=
(\BJ,w\psie)_\Omega - (\curl \BA_h,\curl(w\psie))_\Omega
\\
&=
(\BJ,w\psie)_{\ome} - (\curl \BA_h,\grad w \times \psie + w\curl \psie)_{\ome}
\\
&=
(\BJ \cdot \psie - \curl \BA_h \cdot \curl \psie,w)_{\ome}
-
(\psie \times \curl \BA_h,\grad w)_{\ome},
\end{align*}
one easily sees that
\begin{equation*}
(\grad r_\edge + \psie \times \curl \BA_h,\grad w)_{\ome}
=
(\psie \cdot \BJ-\curl \psie \cdot \curl \BA_h,w)_{\ome}
\end{equation*}
for all $w \in H^1_\star(\ome)$. Then, since \eqref{eq_definition_sige} is the mixed
formulation of \eqref{tmp_definition_r_edge}, we have
$-\sige = \grad r_\ell + \psie \times \grad \BA_h$,
and \eqref{eq_dual_characterization} follows.
\end{proof}

Notice that because of \eqref{eq_galerkin_orthogonality}, for $\edge \in \CE_h$, we do have
\begin{equation}
\label{eq_compatibility_condition}
(\psie \cdot \BJ-\curl \psie \cdot \curl \BA_h,1)_{\ome}
=
(\BJ,\psie)_\Omega-(\curl \BA_h,\curl \psie)_{\Omega}
=
0
\end{equation}
whenever $\psie \not \subset \overline{\GD}$, ensuring the well-posedness of
\eqref{eq_definition_sige}.

\begin{remark}[Broken patchwise equilibration]
\label{remark_broken_estimator}
The broken patchwise equilibration procedure in \cite{chaumontfrelet_ern_vohralik_2021a}
hinges on related localized norms of the residual functional, namely:
\begin{equation*}
\|\BCR\|_{\dagger,\edge}
\eq
\sup_{\substack{\bv \in \BH_0(\ccurl,\ome) \\ \|\curl \bv\| = 1}}
\langle \BCR,\bv \rangle,
\qquad \edge \in \CE_h,
\end{equation*}
where $\BH_0(\ccurl,\ome)$ is defined as $\BH_0(\ddiv,\ome)$ in Section
\ref{section_local_spaces}. Since $v \psie \in \BH_0(\ccurl,\ome)$ with
$\|\curl(v\psie)\|_{\ome} \leq \Cconte \|\grad v\|_{\ome}$ when $v \in H^1_\star(\ome)$,
we have
\begin{equation*}
\|\BCR\|_{\star,\edge} \leq \Cconte \|\BCR\|_{\dagger,\edge}.
\end{equation*}
On the other hand, we have from \cite[Lemma 5.5]{chaumontfrelet_ern_vohralik_2021a} that
\begin{equation*}
\|\BCR\|_{\dagger,\edge} \leq \|\curl (\BA-\BA_h)\|_{\ome},
\end{equation*}
where, in comparison to \eqref{eq_continuous_efficiency}, the constant $\Cconte$ is omitted.
As a result, the constant $\Cconte$ is moved from the reliability estimate to the
efficiency estimate here as compared to \cite{chaumontfrelet_ern_vohralik_2021a}.
\end{remark}

\section{An edge-based a posteriori estimator}
\label{section_edge_estimator}

This section introduces a first a posteriori error estimator that is attached to the edges
of the mesh.  It is simply defined by mimicking the definition of $\sige$ in
\eqref{eq_definition_sige} at the discrete level. As a result, for all edges $\edge \in \CE_h$,
we fix a polynomial degree $\qe \geq \pe+1$, and we set
\begin{equation}
\label{eq_definition_sigeh}
\sigeh \eq \arg \min_{\substack{
\bv_h \in \RT_{\qe}(\CTe) \cap \BH_0(\ddiv,\ome)
\\
\div \bv_h = \pi_{\qe}(\psie \cdot \BJ - \curl \psie \cdot \curl \BA_h)
}}
\|\bv_h + \psie \times \curl \BA_h\|_{\ome}.
\end{equation}
as well as
\begin{equation}
\label{eq_definition_estimator}
\eta_\edge \eq \|\sigeh + \psie \times \curl \BA_h\|_{\ome}.
\end{equation}
Notice that \eqref{eq_definition_sigeh} indeed provides a sound definition for $\sigeh$
since the compatibility condition in \eqref{eq_compatibility_condition} holds whenever
required.

Recalling the results of Section \ref{section_localization}, it is clear that $\eta_\edge$ will
make a good estimator if the discrete minimizer $\sigeh$ is sufficiently close to the
continuous minimizer $\sige$ for each $\edge \in \CE_h$. This is classical in the analysis
of equilibrated estimator
\cite{braess_pillwein_schoberl_2009a,chaumontfrelet_ern_vohralik_2021a,ern_vohralik_2018a}
and the corresponding result is often called ``stable discrete minimization''.
The result we actually need is stable discrete minimization in the $\BH(\ddiv)$
Sobolev space over an edge patch $\CTe$. For the sake of shortness, we skip the proof of this result,
as it can be easily obtained by combining the proof of stable discrete minimization in
$\BH(\ddiv)$ over a vertex patch \cite{braess_pillwein_schoberl_2009a,ern_vohralik_2018a},
and the proof of stable discrete minimization in $\BH(\ccurl)$
over an edge patch \cite{chaumontfrelet_ern_vohralik_2020a,chaumontfrelet_ern_vohralik_2021a}.

\begin{proposition}[Stable discrete minimization]
\label{proposition_stable_discrete_minimization}
Consider an edge $\edge \in \CE_h$ and a polynomial degree $q \geq 0$.
Let $r \in \CP_q(\CTe)$ and $\bg \in \RT_q(\CTe)$, and if
$\edge \not \subset \overline{\GD}$, assume that $(r,1)_{\ome} = 1$.
Then, we have
\begin{equation}
\label{eq_stable_discrete_minimization}
\min_{\substack{
\bv_h \in \RT_q(\CTe) \cap \BH_0(\ddiv,\ome)
\\
\div \bv_h = r
}}
\|\bv_h + \bg\|_{\ome}
\leq
\Cste
\min_{\substack{
\bv \in \BH_0(\ddiv,\ome)
\\
\div \bv = r
}}
\|\bv + \bg\|_{\ome}
\end{equation}
where the constant $\Cste$ only depends on $\kappa_\edge$.
\end{proposition}

Remark that since the minimization set in the left-hand side of
\eqref{eq_stable_discrete_minimization} is contained in the
minimization of the right-hand side, the reverse inequality
\begin{equation}
\label{eq_stable_discrete_minimization_reverse}
\min_{\substack{
\bv \in \BH_0(\ddiv,\ome)
\\
\div \bv = r
}}
\|\bv + \bg\|_{\ome}
\leq
\min_{\substack{
\bv_h \in \RT_q(\CTe) \cap \BH_0(\ddiv,\ome)
\\
\div \bv_h = r
}}
\|\bv_h + \bg\|_{\ome}
\end{equation}
trivially holds true.

Let us further comment that since both sides of
\eqref{eq_stable_discrete_minimization} vanish if and only if
$\bg \in \RT_q(\CTe) \cap \BH(\ddiv,\ome)$ with $\div \bg = r$,
by linearity, it is clear that the inequality should hold for
some constant $\Cste$. It is also clear, based on standard scaling
arguments and Piola mappings, that $\Cste$ is independent of the mesh size.
The non-trivial part is two show that $\Cste$ does not depend on $q$.

Before establishing our key reliability and efficiency results, we need
to cope with the fact that the divergence constraints in the continuous
and discrete minimiaztion problems \eqref{eq_definition_sige} and
\eqref{eq_definition_sigeh} defining $\sige$ and $\sigeh$ are different
when $\BJ$ is not a polynomial. Classically, this is done by introducing
an oscillation term. Notice that since $\qe \geq \pe+1$, the estimate in
\eqref{eq_estimate_osce} shows that $\osce$ converges to zero faster than
the error, justifying the ``oscillation'' terminology
(recall from \eqref{eq_scaling_edge_functions} that $\|\psie\|_{L^\infty(\ome)} \simeq 1$).

\begin{lemma}[Data oscillation]
\label{lemma_data_oscillation}
For all $\edge \in \CE_h$, we have
\begin{equation}
\label{eq_data_oscillation}
\left |
\|\BCR\|_{\star,\edge}
-
\min_{\substack{
\bv \in \BH_0(\ddiv,\ome)
\\
\div \bv = \pi_{\qe}(\psie \cdot \BJ-\curl \psie \cdot \curl \BA_h)
}}
\|\bv+\psie \times \curl \BA_h\|_{\ome}
\right |
\leq
\osce
\end{equation}
with
\begin{equation}
\label{eq_definition_osce}
\osce \eq \CPe \he \|\psie \cdot \BJ - \pi_{\CTe,\qe}(\psie \cdot \BJ)\|_{\ome}.
\end{equation}
In addition, if $\BJ \in \BH^{\qe+1}(\CTe)$, we have
\begin{equation}
\label{eq_estimate_osce}
\osce \leq \frac{\CPe \|\psie\|_{L^\infty(\ome)}}{\pi^{\qe}}
\he^{\qe+1}
|\BJ|_{\BH^{\qe+1}(\CTe)}.
\end{equation}
\end{lemma}

\begin{proof}
Let $\ell \in \CE_h$. Following the proof of Theorem \ref{theorem_dual_characterization},
we know that $\|\BCR\|_{\ell,\star} = \|\grad r_\ell\|_{\ome}$ where
$r_\ell$ is the unique element of $H^1_\star(\ome)$ such that
\begin{align*}
(\grad r_\ell,\grad w)_{\ome}
=
(\psie \cdot \BJ-\curl \psie \cdot \curl \BA_h,w)_{\ome}
-
(\psie \times \curl \BA_h,\grad w)_{\ome}
\end{align*}
for all $w \in H^1_\star(\ome)$. Similarly, we have
\begin{equation*}
\min_{\substack{
\bv \in \BH_0(\ddiv,\ome)
\\
\div \bv = \pi_{\qe}(\psie \cdot \BJ-\curl \psie \cdot \curl \BA_h)
}}
\|\bv+\psie \times \curl \BA_h\|_{\ome}
=
\|\grad r^\star\|_{\ome}
\end{equation*}
where $r^\star$ is the unique element of $H^1_\star(\ome)$ such that
\begin{equation*}
(\grad r^\star,\grad w)_{\ome}
=
(\pi_{\qe}(\psie \cdot \BJ - \curl \psie \cdot \curl \BA_h),w)_{\ome}
-
(\psie \times \curl \BA_h,\grad w)_{\ome}
\end{equation*}
for all $w \in H^1_\star(\ome)$. Since $\curl \psie \cdot \curl \BA_h \in \CP_{\qe}(\CTe)$,
it follows that
\begin{align*}
(\grad (r_\edge-r^\star),\grad w)_{\ome}
&=
(\psie \cdot \BJ-\pi_{\CTe,\qe}(\psie \cdot \BJ),w)_{\ome}
\\
&\leq
\CPe \he \|\psie \cdot \BJ-\pi_{\qe} (\psie \cdot \BJ)\|_{\ome}\|\grad w\|_{\ome}.
\end{align*}
for all $w \in H^1_\star(\ome)$. Picking $w = r_\edge-r^\star$, we have
\begin{equation*}
\|\grad(r_\edge-r^\star)\|_{\ome}
\leq
\CPe \he \|\psie \cdot \BJ-\pi_{\qe} (\psie \cdot \BJ)\|_{\ome},
\end{equation*}
and \eqref{eq_data_oscillation} follows from the reverse triangle inequality.

Finally, since $\psie \in \BCP_1(\CTe)$, \eqref{eq_estimate_osce} follows from
\begin{align*}
\|\psie \cdot \BJ - \pi_{\CTe,\qe} (\psie \cdot \BJ)\|_{\ome}
&\leq
\|\psie \cdot \BJ - \psie \cdot (\pi_{\CTe,\qe-1}(\BJ))\|_{\ome}
\\
&\leq
\|\psie\|_{L^\infty(\ome)}\|\BJ - \pi_{\CTe,\qe-1}(\BJ)\|_{\ome}
\\
&\leq
\|\psie\|_{L^\infty(\ome)}\left (\frac{\he}{\pi} \right )^{\qe}|\BJ|_{\BH^{\qe}(\ome)},
\end{align*}
where we have used \eqref{eq_poincare_Kq} together with the fact
that $h_K \leq \he$ for all $K \in \CTe$.
\end{proof}

We are now ready to state the main result of this section, which is a direct consequence
of the three theorems of Section \ref{section_localization}, identity
\eqref{eq_dual_norm_error}, as well as estimates \eqref{eq_stable_discrete_minimization},
\eqref{eq_stable_discrete_minimization_reverse} and \eqref{eq_data_oscillation}.

\begin{theorem}[Edge-based error estimator]
\label{theorem_edge_estimator}
The estimates
\begin{equation}
\label{eq_discrete_reliability}
\|\curl(\BA-\BA_h)\|_\Omega
\leq
\sqrt{6} \Clift \left (\sum_{\edge \in \CE_h} (\eta_\edge^2 + \osce^2) \right )^{1/2}
\end{equation}
and
\begin{equation}
\label{eq_discrete_efficiency}
\eta_\edge \leq \Cste \Cconte \|\curl(\BA-\BA_h)\|_{\ome}  + \Cconte \osce
\qquad
\forall \edge \in \CE_h
\end{equation}
hold true.
\end{theorem}

\begin{remark}[An interpretation of what $\eta_\edge$ measures]
Assume that $\BJ \in \RT_{\pe}(\CTe)$ and that $\eta_\edge = 0$ for an edge $\edge \in \CE_h$.
Then, we have
\begin{equation*}
\psie \cdot \BJ - \curl \psie \cdot \curl \BA_h
=
\div (\psie \times \curl \BA_h)
=
-\curl \psie \cdot \curl \BA_h + \psie \cdot \curl \curl \BA_h,
\end{equation*}
meaning that $\psie \cdot (\curl \curl \BA_h-\BJ) = 0$. We then see that
$\BA_h$ locally satisfies the curl-curl problem strongly in $\ome$, at
least in the $\taue$ direction.
\end{remark}

\begin{remark}[Alternate construction by sequential sweeps]
We can decrease the computational cost for constructing $\sigeh$
by following the approach presented in \cite[Theorem 3.2]{chaumontfrelet_ern_vohralik_2021a}.
In this case, instead of solving the patchwise minimization in \eqref{eq_definition_sigeh},
we construct an alternative local contribution $\sigeh^\star$ by sweeping through the edge
patch. In this approach, for each $K \in \CTe$, $\sigeh^\star|_K$ is defined through an
elementwise minimization problem in $K$, similar to \eqref{eq_definition_sigeh}.
\end{remark}

\begin{remark}[$p$-adaptive reconstruction]
Let $\edge \in \CE_h$. In our presentation, we employed the assumption that
$\BA_h \in \ND_{\pe}(\CTe)$. It is possible to further take advantage of the fact
that $\BA_h \in \ND_{p_K}(K)$ for all $K \in \CTe$ for some $0 \leq p_K \leq \pe$.
Indeed, instead of seeking $\sigeh$ in the space $\RT_{\qe}(\CTe)$ (recall that
$\qe \geq \pe+1$), we can instead require that $\sigeh \in \RT_{p_K+1}(K)$ for all
$K \in \CTe$. In doing so, it is still true that \eqref{eq_stable_discrete_minimization}
holds true, but with a constant $\Cste$ that may depend on the distribution $\{p_K\}_{K \in \CTe}$.
Indeed, it is not clear wether constant $\Cste$ is polynomial-degree-robust in this case as,
in particular, the proof techniques employed in
\cite{braess_pillwein_schoberl_2009a,chaumontfrelet_ern_vohralik_2021a,ern_vohralik_2018a}
fail in this case.
\end{remark}

\section{Prager--Synge type estimates}
\label{section_equilibrated_estimator}

In this section, we elaborate a second error estimator. It hinges on a
new Prager--Synge type identity and a corresponding discrete equilibration
procedure. We actually show that the local contributions $\sigeh$ previously
introduced to build the edge-based estimator can be recombined to provide
equilibrated fields.

We begin with our new Prager--Synge type identity.

\begin{theorem}[Prager--Synge identity]
\label{theorem_prager_synge}
Assume that, for $1 \leq k \leq 3$, $\BS^k \in \BH_{\GN}(\ddiv,\Omega)$,
then we have
\begin{multline}
\label{eq_prager_synge}
(\curl(\BA-\BA_h),\curl \bv)_\Omega
=
\\
\sum_{k = 1}^3 \left \{
(\BJ_k-\div \BS^k,\bv_k)_\Omega
- 
(\BS^k + \cank \times \curl \BA_h,\grad \bv_k)_\Omega
\right \}
\end{multline}
for all $\bv \in \BH_{\GD}^1(\Omega)$. In addition, if
\begin{equation}
\label{eq_orthogonality_prager_synge}
(\div \BS^k-\BJ_k,r_h)_\Omega = 0 \quad \forall r_h \in \CP_0(\CT_h),
\end{equation}
then,
\begin{multline}
\label{eq_estimate_prager_synge}
\|\curl(\BA-\BA_h)\|_\Omega
\leq
\\
\Clift
\left (
\sum_{K \in \CT_h}
\sum_{k=1}^3
\left (\frac{h_K}{\pi}\|\BJ_k-\div \BS^k\|_K + \|\BS^k + \cank \times \curl \BA_h\|_K\right )^2
\right )^{1/2}.
\end{multline}
\end{theorem}

\begin{proof}
Let $\bv \in \BH^1_{\GD}(\Omega)$. We have
\begin{equation}
\label{tmp_ps1}
(\curl(\BA-\BA_h),\curl \bv)_\Omega
=
(\BJ,\bv)_\Omega-(\curl \BA_h,\curl \bv)_\Omega.
\end{equation}
On the one hand, we have
\begin{equation}
\label{tmp_ps2}
(\BJ,\bv)_\Omega
=
\sum_{k = 1}^3 (\BJ_k,\bv_k)_\Omega
=
\sum_{k = 1}^3
\left \{
(\BJ_k-\div \BS_k,\bv_k)_\Omega
-
(\BS_k,\grad \bv_k)_\Omega
\right \}
\end{equation}
On the other hand, it holds that
\begin{equation*}
\curl \bv
=
\sum_{k=1}^d \curl (\bv_k \cank)
=
\sum_{k=1}^d \grad \bv_k \times \cank,
\end{equation*}
leading to
\begin{align}
\label{tmp_ps3}
(\curl \BA_h,\curl \bv)_\Omega
&=
\sum_{k=1}^d (\curl \BA_h,\grad \bv_k \times \cank)_\Omega
\\
&=
\nonumber
\sum_{k=1}^d (\cank \times \curl \BA_h,\grad \bv_k)_\Omega.
\end{align}
Identity \eqref{eq_prager_synge} then easily follows from
\eqref{tmp_ps1}, \eqref{tmp_ps2} and \eqref{tmp_ps3}.
Finally, estimate \eqref{eq_estimate_prager_synge} is a direct
consequence of \eqref{eq_prager_synge}, orthogonality
property \eqref{eq_orthogonality_prager_synge} together
with elemenwise Poincar\'e inequality \eqref{eq_poincare_K}
and the discussion of Section \ref{section_regular_decomposition},
since
\begin{equation*}
(\BJ_k-\div \BS^k,\bv_k)_K
=
(\BJ_k-\div \BS^k,\bv_k-\pi_{K,0} \bv_k)_K
\leq
\frac{h_K}{\pi} \|\BJ-\div \BS^k\|_K\|\grad \bv_k\|_K.
\end{equation*}
\end{proof}

\begin{remark}[The constant $\Clift$]
\label{remark_clift}
The constant $\Clift$ is the price we pay for working with a test function in
$\BH^1_{\GD}(\Omega)$ instead of $\BH_{\GD}(\ccurl,\Omega)$. Specifically,
starting from a general function $\bv \in \BH_{\GD}(\ccurl,\Omega)$, we
introduce a function $\bw \in \BH^1_{\GD}(\Omega)$ such that $\curl \bv = \curl \bw$.
When doing this operation, there is no reason to think that
$\|\grad \bw\|_\Omega \leq \|\curl \bv\|_\Omega$ in the general case, hence the
need for the constant $\Clift$. However, for our purposes, we do not need to impose
that $\curl \bv = \curl \bw$. Indeed, we simply need that
$(\curl(\BA-\BA_h),\curl \bv)_\Omega = (\curl(\BA-\BA_h),\curl \bw)$, which is a much
less demanding condition. Although the author is not currently aware of a way to take
advantage of this idea, it hints toward the fact that a sharper estimate without the
constant $\Clift$ may be achieved. We further point out that this observation is not
limited to the particular estimators considered in this work, but to general a posteriori
estimators using regular decompositions
\cite{chaumontfrelet_ern_vohralik_2021a,nicaise_creuse_2003a,schoberl_2008a}.
A similar remark also holds true for the upper bounds of Theorems
\ref{theorem_continuous_reliability} and \ref{theorem_edge_estimator}.
\end{remark}

\begin{remark}[Improved oscillation term]
Under the additional assumption that $(\div \BS^k-\BJ_k,r_h)_\Omega = 0$
for all $r_h \in \CP_q(\CT_h)$ for some $q \geq 0$, it is possible to improve the
factor $h_K/\pi$ in the oscillation term to $Ch_K/q$ (see, e.g.,
\cite[Lemma 4.1]{babuska_suri_1987a}). For the sake of simplicity though, we only focus
on the simpler version, in particular because the constant appearing in the improved version is
not easily computable in practice.
\end{remark}

The next step is to provide a discrete construction of fields
$\BS^k_h$ satisfying the requirements of of Theorem \ref{theorem_prager_synge}.
As previously advertised, this is easily achieved by recombining the local
contributions $\sigeh$ introduced at \eqref{eq_definition_sigeh}. We thus set,
for $1 \leq k \leq 3$, the field
\begin{equation}
\BS^k_h \eq \sum_{\edge \in \CE_h} (\taue \cdot \cank) \sigeh.
\end{equation}

\begin{lemma}[Discrete equilibration]
\label{lemma_equilibration}
For $1 \leq k \leq 3$, we have $\BS_h^k \in \BH_{\GN}(\ddiv,\Omega)$ with
\begin{equation}
\label{eq_discrete_equilibration}
(\div \BS_h^k - \BJ_k,r_h)_\Omega = 0 \quad \forall r_h \in \CP_0(\CT_h).
\end{equation}
\end{lemma}

\begin{proof}
Let $k \in \{1,2,3\}$. That $\BS_h^k \in \BH_{\GN}(\ddiv,\Omega)$
is a direct consequence of the fact that $\sigeh \in \BH_0(\ddiv,\ome)$
for all $\edge \in \CE_h$. Then, consider an element $K \in \CT_h$ and $r_h \in \CP_0(K)$.
Recalling \eqref{eq_definition_sigeh}, we see that
\begin{align*}
(\div \BS_h^K,r_h)_K
&=
\sum_{\edge \in \CE_h} (\taue \cdot \cank) (\div \sigeh,r_h)_K
\\
&=
\sum_{\edge \in \CE_h} (\taue \cdot \cank) \left (\pi_{\qe}
\left (\psie \cdot \BJ - \curl \psie \cdot \curl \BA_h\right ),r_h\right )_K
\\
&=
\sum_{\edge \in \CE_h}
(\taue \cdot \cank) \left (\psie \cdot \BJ - \curl \psie \cdot \curl \BA_h,r_h\right )_K.
\end{align*}
By linearity, and recalling \eqref{eq_partition_of_unity}, we have
\begin{equation*}
(\div \BS_h^k,r_h)_K
=
(\cank \cdot \BJ - \curl \cank \cdot \curl \BA_h,r_h)_K
=
(\BJ_k,r_h)_K,
\end{equation*}
leading to \eqref{eq_discrete_equilibration}.
\end{proof}

Having introduced the equilibrated fields $\BS_h^k$, we simply
define our estimator with the elementwise contributions in the right-hand
side of \eqref{eq_estimate_prager_synge}. We thus set
\begin{equation}
\label{eq_definition_etaK}
\eta_K^k \eq \|\cank \times \curl \BA_h + \BS_h^k\|_K
\end{equation}
for each element $K \in \CT_h$ and $k \in \{1,2,3\}$. As we establish below,
this estimator is reliable and efficient.

\begin{theorem}[Equilibrated estimator]
\label{theorem_equilibrated_estimator}
The following upper bound
\begin{equation}
\label{eq_equilibrated_reliability}
\|\curl(\BA-\BA_h)\|_\Omega
\leq
\Clift
\left (
\sum_{K \in \CT_h} \sum_{k=1}^3\left (
\eta_K^k + \oscK^k
\right )^2
\right )^{1/2}
\end{equation}
holds true with
\begin{equation}
\oscK^k \eq \frac{h_K}{\pi}\|\div \BS_h^k - \BJ_k\|_K \qquad \forall K \in \CT_h,
\quad
1 \leq k \leq 3.
\end{equation}
In addition, we have the lower bounds
\begin{equation}
\label{eq_equilibrated_efficiency}
\sum_{k=1}^3 (\eta_K^k)^2
\leq
6 \sum_{\edge \in \CE_K} \left (\Cste \Cconte \|\curl(\BA-\BA_h)\|_{\ome} + \Cconte \osce\right )^2
\end{equation}
for all $K \in \CT_h$.
\end{theorem}

\begin{proof}
On the one hand, \eqref{eq_equilibrated_reliability} is a direct
consequence of Theorem \ref{theorem_prager_synge} and Lemma \ref{lemma_equilibration}.
On the other hand, we have
\begin{align*}
\|\BS_h^k + \cank \times \curl \BA_h\|_K
&=
\left \|
\sum_{\edge \in \CE_K} (\taue \cdot \cank) \sigeh
+
\left (
\sum_{\edge \in \CE_K} (\cank \cdot \taue) \psie
\right )
\times \curl \BA_h
\right \|_K
\\
&\leq
\sum_{\edge \in \CE_K} |\cank \cdot \taue| \cdot \|\sigeh + \psie \times \curl \BA_h\|_K
\\
&\leq
\sum_{\edge \in \CE_K} |\cank \cdot \taue| \cdot \|\sigeh + \psie \times \curl \BA_h\|_{\ome}
\\
&=
\sum_{\edge \in \CE_K} |\cank \cdot \taue| \eta_\ell,
\end{align*}
and \eqref{eq_equilibrated_efficiency} follows from \eqref{eq_discrete_efficiency}
using Cauchy-Schwarz inequality
\begin{multline*}
\sum_{k=1}^3 (\eta_K^k)^2
=
\sum_{k=1}^3 \|\BS_h^k+\cank \times \curl \BA_h\|_K^2
\\
\leq
\sum_{k=1}^3 \left (\sum_{\edge \in \CE_K} |\cank \cdot \taue| \eta_\edge\right )^2
\leq
6
\sum_{\edge \in \CE_K} \sum_{k=1}^3 |\cank \cdot \taue|^2 \eta_\edge^2
\end{multline*}
and observing that
\begin{equation*}
\sum_{k=1}^3 |\cank \cdot \taue|^2 = |\taue|^2 = 1.
\end{equation*}
\end{proof}

%% TODO: remark on oscillation term

\section{Numerical examples}
\label{section_numerical_examples}

\subsection{Settings} This section presents a set numerical examples.
We first describe the general setting.

\subsubsection{Estimators}

For the sake of shortness, we disregard all the oscillations terms. That way, we can simply set
\begin{equation*}
\eta_K^2 \eq \sum_{k=1}^3 (\eta_K^k)^2,
\end{equation*}
for all $K \in \CT_h$ and
\begin{equation*}
\eta_{\rm edge}^2 \eq 6\sum_{\edge \in \CE_h} \eta_\edge^2,
\qquad
\eta_{\rm cell}^2 \eq \sum_{K \in \CT_h} \eta_K^2.
\end{equation*}

\subsubsection{Discrete solution}

The discrete solution $\BA_h$ is computed using N\'ed\'elec elements
of uniform degree $p \geq 0$ through the usual Galerkin formulation.
For the sake of simplicity, we only consider cases where $\Omega$
is simply connected and $\GD = \partial \Omega$, so that
$\BLba_{\GD}(\Omega) = \grad\left( H^1_{\GD}(\Omega)\right )$. This leads to the definition
of $\BA_h$ has the unique element of $\ND_p(\CT_h) \cap \BH_{\GD}(\ccurl,\Omega)$
such that
\begin{equation*}
(\curl \BA_h,\curl \bv_h)_\Omega = (\BJ,\bv_h)_\Omega
\qquad
(\BA_h,\grad q_h)_\Omega = 0
\end{equation*}
for all $\bv_h \in \ND_p(\CT_h) \cap \BH_{\GD}(\ccurl,\Omega)$ and
$q_h \in \CP_{p+1}(\CT_h) \cap H^1_{\GD}(\Omega)$. Remark in particular
that $\BA_h$ satisfies the assumptions of Section \ref{section_discrete_solution}
with $\pe = p$ for all $\edge \in \CE_h$. The estimator is then computed
with the lowest polynomial degree possible, namely $\qe = q \eq p+1$ for all $\edge \in \CE_h$.

\subsubsection{Error evaluation}

In the first two examples, the analytic solution $\BA$ to the problem is
available, so that we can readily compute the true error
$\err \eq \|\curl(\BA-\BA_h)\|_\Omega$ for comparison purposes. In the last experiment however,
we do not have access to $\BA$, and if $\BA_h \in \ND_p(\CT_h) \cap \BH_{\GD}(\Omega)$
we assess the error using the quantity $\err \eq \|\curl(\widetilde \BA_h-\BA_h)\|_\Omega$,
where $\widetilde \BA_h \in \ND_{p+2}(\CT_h) \cap \BH_{\GD}(\Omega)$ is computed using
the same mesh than $\BA_h$, but with a higher polynomial degree.

\subsubsection{Mesh generation}

The meshes we employ are generated with with \gmsh \cite{geuzaine_remacle_2009a}
and \mmg \cite{dobrzynski_2012a}. Starting from a file ``{\tt geom.geo}'' describing
the geometry of the experiment, we generate a first tetrahedral mesh using the command
``{\tt gmsh -3 geom.geo -format mesh}''. This generates a mesh stored in ``{\tt geom.mesh}''
that is then passed to \mmg for further refinements. Specifically, when we talk about a
``mesh of size $\hmax$'' when it has been generated using the command
``{\tt mmg3D -in geom.mesh -out mesh.mesh -hmax $\hmax$}'', which generates two files:
``{\tt mesh.mesh}'' describing the mesh, and ``{\tt mesh.sol}'' describing the mesh size
around each vertex.

We also employ \mmg for iterative mesh refinements. In this case, based on
finite element computation with the mesh described by the files ``{\tt mesh.mesh}''
and ``{\tt mesh.sol}'' we create a new mesh the following way:
\begin{enumerate}
\item Find an ordering $\iota: \{1,\dots,\sharp \CE_h\} \to \CE_h$ such that
$j \to \eta_{\iota(j)}$ is decreasing.
\item Select the smallest integer $m$ such that
\begin{equation*}
\sum_{j=1}^m \eta_{\iota(j)}^2 \leq \theta \sum_{\edge \in \CE_h} \eta_\edge^2,
\qquad
\theta \eq 0.1.
\end{equation*}
\item Mark all the vertices associated with the edges $\iota(j)$, $j \in \{1,\dots,m\}$.
\item Generate a new file ``{\tt refinement.sol}'' where the mesh size associated with
all the marked vertices is divided by two.
\item Generate the refined mesh with the command
``{\tt mmg3D -in mesh.mesh -sol refinement.sol -out refined\_mesh.mesh -hgrad 10.}''.
\end{enumerate}
We also employ the equilibrated estimator $\{\eta_K\}_{K \in \CT_h}$ instead of
$\{\eta_\ell\}_{\edge \in \CE_h}$. In this case, we follow the same procedure with
edges replaced by cells.

\subsubsection{Quantities of interest and legends}

Throughout this section, we will focus on 5 quantities of interest, namely:
the error $\err$, the summed estimators
$\eta_{\rm edge}$ and $\eta_{\rm cell}$, and the effectivity indices
$\effe \eq \eta_{\rm edge}/\err$ and $\effc \eq \eta_{\rm cell}/\err$.
These quantities will be represented on two kinds of figures.
On the one hand, we will display the error, and summed estimator
on ``Error and estimators'' plot with the following legend:
\\
\begin{tikzpicture}
\draw[white] (-2,-1) rectangle (0,1);

\draw[black,solid] (0,0) -- node[anchor=north]{\color{black} $\err$} (2,0);

\draw[blue,dotted] (3,0) -- node[anchor=north]{\color{black} $\eta_{\rm edge}$} (5,0);

\draw[red,dashed] (6,0) -- node[anchor=north]{\color{black} $\eta_{\rm cell}$} (8,0);

\draw[black,fill=black] (1-.075,0-.075) -- (1+.075,0-.075) -- (1.,.075) -- cycle;

\draw[blue,fill=blue] (3.925,-.075) rectangle (4.075,.075);

\draw[red,fill=red] (7,0) circle(.075);
\end{tikzpicture}
\\
On the other hand, ``Effectivity indices'' figures have the following legend:
\\
\\
\begin{tikzpicture}
\draw[white] (-2,-1) rectangle (0,1);

\draw[blue,dotted] (3,0) -- node[anchor=north]{\color{black} $\effe$} (5,0);

\draw[red,dashed] (6,0) -- node[anchor=north]{\color{black} $\effc$} (8,0);

\draw[blue,fill=blue] (3.925,-.075) rectangle (4.075,.075);

\draw[red,fill=red] (7,0) circle(.075);
\end{tikzpicture}

\subsection{Smooth solution in a cube}

We start by considering the unit cube $\Omega \eq (0,1)^3$
together with the load term
\begin{equation*}
\BJ
\eq
(3\pi)^2
\left (
\begin{array}{c}
 \cos(\pi\bx_1)\sin(\pi\bx_2)\sin(\pi\bx_3)
\\
-\sin(\pi\bx_1)\cos(\pi\bx_2)\sin(\pi\bx_3)
\\
0
\end{array}
\right )
\end{equation*}
One readily sees that $\div \BJ = 0$ and that the associated solution is
\begin{equation*}
\BA
\eq
\left (
\begin{array}{c}
 \cos(\pi\bx_1)\sin(\pi\bx_2)\sin(\pi\bx_3)
\\
-\sin(\pi\bx_1)\cos(\pi\bx_2)\sin(\pi\bx_3)
\\
0
\end{array}
\right ).
\end{equation*}

We first consider a ``$h$-convergence'' example where we fix the polynomial degree $p$,
and consider a sequence of uniform meshes generated by \mmg with $\hmax \eq 1,1/2,1/4,1/16$
and $1/32$. Figure \ref{figure_cube_hconv} presents the results. As can be seen from the
left-panel, the expected convergence rate in $\CO(h^{p+1})$ is achieved. The right-panel
illustrates that both estimators are reliable and efficient, since the efficiency indices
are grater that one and independent of $h$.

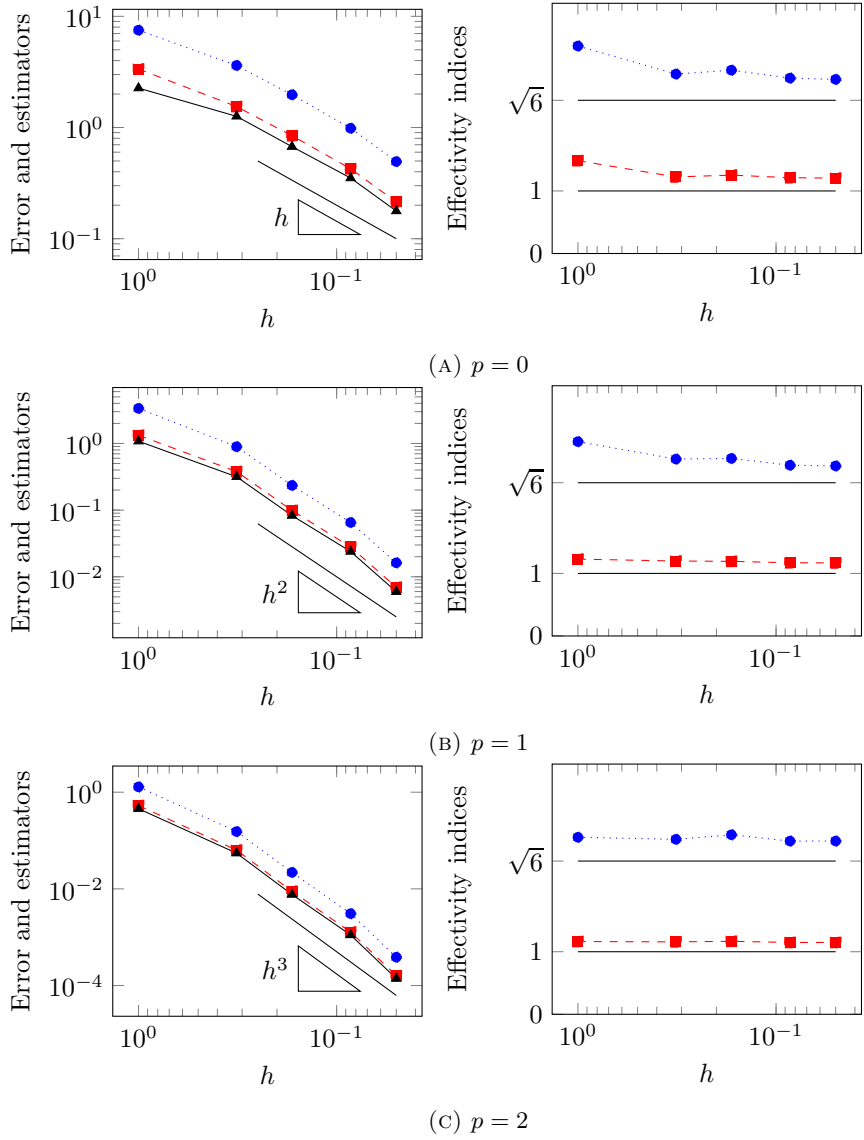
\begin{figure}
\begin{minipage}{\linewidth}
\begin{minipage}{.45\linewidth}
\begin{tikzpicture}
\begin{axis}
[
	width  = \linewidth,
	xlabel = {$h$},
	ylabel = {Error and estimators},
	ymode  = log,
	xmode  = log,
	x dir  = reverse
]

\plot[mark=*,color=blue,dotted]
table[x expr = \thisrow{h}/1000, y expr = sqrt(6)*\thisrow{etae}]
{figures/cube/data/P0.txt};

\plot[mark=square*,color=red,dashed]
table[x expr = \thisrow{h}/1000, y = etac]
{figures/cube/data/P0.txt};

\plot[mark=triangle*,color=black,solid]
table[x expr = \thisrow{h}/1000, y = err]
{figures/cube/data/P0.txt};

\plot[solid,black,domain=0.05:.25] {2*x};

\SlopeTriangle{.6}{-.2}{.1}{-1}{$h$}{}

\end{axis}
\end{tikzpicture}
\end{minipage}
\begin{minipage}{.45\linewidth}
\begin{tikzpicture}
\begin{axis}
[
	width  = \linewidth,
	xlabel = {$h$},
	ylabel = {Effectivity indices},
	xmode  = log,
	x dir  = reverse,
	ymin   = 0,
	ymax   = 4,
	ytick = {0,1,2.44948974278317809819},
	yticklabels={0,1,$\sqrt 6$}
]

\plot[mark=*,color=blue,dotted]
table[x expr = \thisrow{h}/1000, y expr = sqrt(6)*\thisrow{etae}/\thisrow{err}]
{figures/cube/data/P0.txt};

\plot[mark=square*,color=red,dashed]
table[x expr = \thisrow{h}/1000, y expr = \thisrow{etac}/\thisrow{err}]
{figures/cube/data/P0.txt};

\plot[solid,black,domain=0.05:1] {1};
\plot[solid,black,domain=0.05:1] {2.44948974278317809819};

\end{axis}
\end{tikzpicture}
\end{minipage}
\subcaption{$p=0$}
\end{minipage}

\begin{minipage}{\linewidth}
\begin{minipage}{.45\linewidth}
\begin{tikzpicture}
\begin{axis}
[
	width  = \linewidth,
	xlabel = {$h$},
	ylabel = {Error and estimators},
	ymode  = log,
	xmode  = log,
	x dir  = reverse
]

\plot[mark=*,color=blue,dotted]
table[x expr = \thisrow{h}/1000, y expr = sqrt(6)*\thisrow{etae}]
{figures/cube/data/P1.txt};

\plot[mark=square*,color=red,dashed]
table[x expr = \thisrow{h}/1000, y = etac]
{figures/cube/data/P1.txt};

\plot[mark=triangle*,color=black,solid]
table[x expr = \thisrow{h}/1000, y = err]
{figures/cube/data/P1.txt};

\plot[solid,black,domain=0.05:.25] {x*x};

\SlopeTriangle{.6}{-.2}{.1}{-2}{$h^2$}{}

\end{axis}
\end{tikzpicture}
\end{minipage}
\begin{minipage}{.45\linewidth}
\begin{tikzpicture}
\begin{axis}
[
	width  = \linewidth,
	xlabel = {$h$},
	ylabel = {Effectivity indices},
	xmode  = log,
	x dir  = reverse,
	ymin   = 0,
	ymax   = 4,
	ytick = {0,1,2.44948974278317809819},
	yticklabels={0,1,$\sqrt 6$}
]

\plot[mark=*,color=blue,dotted]
table[x expr = \thisrow{h}/1000, y expr = sqrt(6)*\thisrow{etae}/\thisrow{err}]
{figures/cube/data/P1.txt};

\plot[mark=square*,color=red,dashed]
table[x expr = \thisrow{h}/1000, y expr = \thisrow{etac}/\thisrow{err}]
{figures/cube/data/P1.txt};

\plot[solid,black,domain=0.05:1] {1};
\plot[solid,black,domain=0.05:1] {2.44948974278317809819};

\end{axis}
\end{tikzpicture}
\end{minipage}
\subcaption{$p=1$}
\end{minipage}

\begin{minipage}{\linewidth}
\begin{minipage}{.45\linewidth}
\begin{tikzpicture}
\begin{axis}
[
	width  = \linewidth,
	xlabel = {$h$},
	ylabel = {Error and estimators},
	ymode  = log,
	xmode  = log,
	x dir  = reverse
]

\plot[mark=*,color=blue,dotted]
table[x expr = \thisrow{h}/1000, y expr = sqrt(6)*\thisrow{etae}]
{figures/cube/data/P2.txt};

\plot[mark=square*,color=red,dashed]
table[x expr = \thisrow{h}/1000, y = etac]
{figures/cube/data/P2.txt};

\plot[mark=triangle*,color=black,solid]
table[x expr = \thisrow{h}/1000, y = err]
{figures/cube/data/P2.txt};

\plot[solid,black,domain=0.05:.25] {.5*x*x*x};

\SlopeTriangle{.6}{-.2}{.1}{-3}{$h^3$}{}

\end{axis}
\end{tikzpicture}
\end{minipage}
\begin{minipage}{.45\linewidth}
\begin{tikzpicture}
\begin{axis}
[
	width  = \linewidth,
	xlabel = {$h$},
	ylabel = {Effectivity indices},
	xmode  = log,
	x dir  = reverse,
	ymin   = 0,
	ymax   = 4,
	ytick = {0,1,2.44948974278317809819},
	yticklabels={0,1,$\sqrt 6$}
]

\plot[mark=*,color=blue,dotted]
table[x expr = \thisrow{h}/1000, y expr = sqrt(6)*\thisrow{etae}/\thisrow{err}]
{figures/cube/data/P2.txt};

\plot[mark=square*,color=red,dashed]
table[x expr = \thisrow{h}/1000, y expr = \thisrow{etac}/\thisrow{err}]
{figures/cube/data/P2.txt};

\plot[solid,black,domain=0.05:1] {1};
\plot[solid,black,domain=0.05:1] {2.44948974278317809819};

\end{axis}
\end{tikzpicture}
\end{minipage}
\subcaption{$p=2$}
\end{minipage}

\caption{Smooth solution example: $h$-convergence}
\label{figure_cube_hconv}
\end{figure}

Next, we perform a ``$p$-convergence'' study where $\hmax$ is fixed, but $p$
increases from $0$ to $6$. The expected exponential convergence rate is observed
on the left-panel of Figure \ref{figure_cube_pconv}. The right-panel of
Figure \ref{figure_cube_pconv} shows that both estimators are reliable and
efficient. This experiment further highlights the $p$-robustness of the
estimator, since the effectivity indices are indeed independent of $p$.

\begin{figure}
\begin{minipage}{\linewidth}
\begin{minipage}{.45\linewidth}
\begin{tikzpicture}
\begin{axis}
[
	width  = \linewidth,
	xlabel = {$p$ ($\hmax=1)$},
	ylabel = {Error and estimators},
	ymode  = log,
]

\plot[mark=*,color=blue,dotted]
table[x = p, y expr = sqrt(6)*\thisrow{etae}]
{figures/cube/data/H1000.txt};

\plot[mark=square*,color=red,dashed]
table[x = p, y = etac]
{figures/cube/data/H1000.txt};

\plot[mark=triangle*,color=black,solid]
table[x = p, y = err]
{figures/cube/data/H1000.txt};

\plot [domain=0:6] {.25*exp(-x)};

\SlopeTriangle{.6}{-.2}{.1}{-1}{$e^{-p}$}{}

\end{axis}
\end{tikzpicture}
\end{minipage}
\begin{minipage}{.45\linewidth}
\begin{tikzpicture}
\begin{axis}
[
	width  = \linewidth,
	xlabel = {$p$ ($\hmax=1)$},
	ylabel = {Effectivity indices},
	ymin   = 0,
	ymax   = 4,
	ytick = {0,1,2.44948974278317809819},
	yticklabels={0,1,$\sqrt 6$}
]

\plot[mark=*,color=blue,dotted]
table[x = p, y expr = sqrt(6)*\thisrow{etae}/\thisrow{err}]
{figures/cube/data/H1000.txt};

\plot[mark=square*,color=red,dashed]
table[x = p, y expr = \thisrow{etac}/\thisrow{err}]
{figures/cube/data/H1000.txt};

\plot[solid,black,domain=0:6] {1};
\plot[solid,black,domain=0:6] {2.44948974278317809819};

\end{axis}
\end{tikzpicture}
\end{minipage}
\subcaption{$\hmax=1$}
\end{minipage}

\begin{minipage}{\linewidth}
\begin{minipage}{.45\linewidth}
\begin{tikzpicture}
\begin{axis}
[
	width  = \linewidth,
	xlabel = {$p$},
	ylabel = {Error and estimators},
	ymode  = log,
]

\plot[mark=*,color=blue,dotted]
table[x = p, y expr = sqrt(6)*\thisrow{etae}]
{figures/cube/data/H0500.txt};

\plot[mark=square*,color=red,dashed]
table[x = p, y = etac]
{figures/cube/data/H0500.txt};

\plot[mark=triangle*,color=black,solid]
table[x = p, y = err]
{figures/cube/data/H0500.txt};

\plot [domain=0:6] {.25*exp(-2*x)};

\SlopeTriangle{.6}{-.2}{.1}{-2}{$e^{-2p}$}{}

\end{axis}
\end{tikzpicture}
\end{minipage}
\begin{minipage}{.45\linewidth}
\begin{tikzpicture}
\begin{axis}
[
	width  = \linewidth,
	xlabel = {$p$},
	ylabel = {Effectivity indices},
	ymin   = 0,
	ymax   = 4,
	ytick = {0,1,2.44948974278317809819},
	yticklabels={0,1,$\sqrt 6$}
]

\plot[mark=*,color=blue,dotted]
table[x = p, y expr = sqrt(6)*\thisrow{etae}/\thisrow{err}]
{figures/cube/data/H0500.txt};

\plot[mark=square*,color=red,dashed]
table[x = p, y expr = \thisrow{etac}/\thisrow{err}]
{figures/cube/data/H0500.txt};

\plot[solid,black,domain=0:6] {1};
\plot[solid,black,domain=0:6] {2.44948974278317809819};

\end{axis}
\end{tikzpicture}
\end{minipage}
\subcaption{$\hmax=1/2$}
\end{minipage}

\begin{minipage}{\linewidth}
\begin{minipage}{.45\linewidth}
\begin{tikzpicture}
\begin{axis}
[
	width  = \linewidth,
	xlabel = {$p$},
	ylabel = {Error and estimators},
	ymode  = log,
]

\plot[mark=*,color=blue,dotted]
table[x = p, y expr = sqrt(6)*\thisrow{etae}]
{figures/cube/data/H0250.txt};

\plot[mark=square*,color=red,dashed]
table[x = p, y = etac]
{figures/cube/data/H0250.txt};

\plot[mark=triangle*,color=black,solid]
table[x = p, y = err]
{figures/cube/data/H0250.txt};

\plot [domain=0:6] {.25*exp(-3*x)};

\SlopeTriangle{.6}{-.2}{.1}{-3}{$e^{-3p}$}{}

\end{axis}
\end{tikzpicture}
\end{minipage}
\begin{minipage}{.45\linewidth}
\begin{tikzpicture}
\begin{axis}
[
	width  = \linewidth,
	xlabel = {$p$},
	ylabel = {Effectivity indices},
	ymin   = 0,
	ymax   = 4,
	ytick = {0,1,2.44948974278317809819},
	yticklabels={0,1,$\sqrt 6$}
]

\plot[mark=*,color=blue,dotted]
table[x = p, y expr = sqrt(6)*\thisrow{etae}/\thisrow{err}]
{figures/cube/data/H0250.txt};

\plot[mark=square*,color=red,dashed]
table[x = p, y expr = \thisrow{etac}/\thisrow{err}]
{figures/cube/data/H0250.txt};

\plot[solid,black,domain=0:6] {1};
\plot[solid,black,domain=0:6] {2.44948974278317809819};

\end{axis}
\end{tikzpicture}
\end{minipage}
\subcaption{$\hmax=1/4$}
\end{minipage}

\caption{Smooth solution example: $p$-convergence}
\label{figure_cube_pconv}
\end{figure}
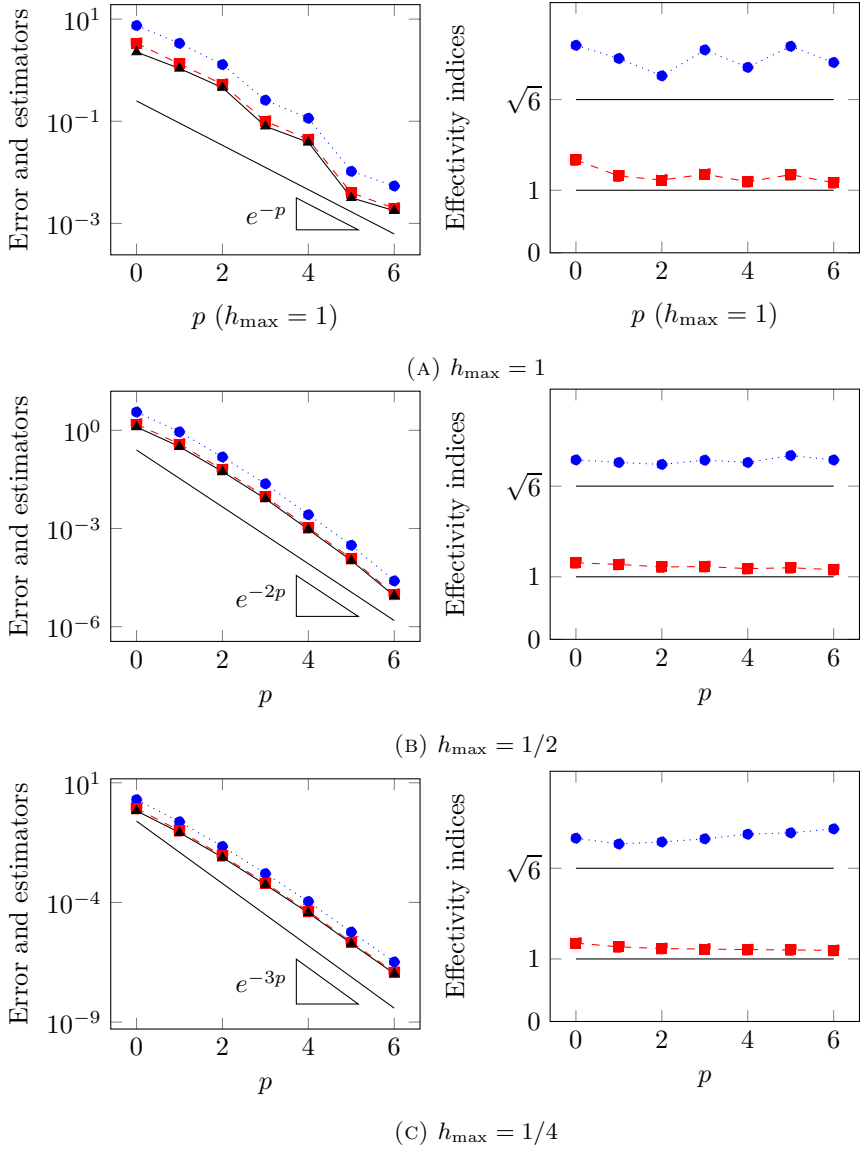

We further note that the effectivity index of the equilibrated estimator
is fairly close to one. Notice that the fact that the effectivity is greater
than one is ensured by Theorem \ref{theorem_equilibrated_estimator}, since the
domain is convex in this example, and we can take $\Clift = 1$ in
\eqref{eq_equilibrated_reliability}. For the edge-based estimator,
the effectivity index is close to, and always greater than, $\sqrt{6}$.
This may hint that the factor $\sqrt{6}$ in \eqref{eq_discrete_reliability}
is spurious, although the author is not aware of a way to suppress it
from a theoretical stand point.

\subsection{Edge singularity in L-type domains}

Given an angle $\phi \in (0,2\pi)$, we consider a domain of the form
$\Omega \eq L \times (0,1)$, where
\begin{equation*}
L \eq \{ \bx = r(\cos(\theta),\sin(\theta)) \in \mathbb R^2; \; |\bx_1|,|\bx_2| < 1,
\quad 0 < \theta < 2\pi-\phi \}.
\end{equation*}
We will in particular consider the cases $\phi=3\pi/4,\pi/2$ and $\pi/8$.
The top faces of these domains are depicted on Figure \ref{figure_ltype_domains}.

\begin{figure}
\begin{minipage}{.30\linewidth}
\begin{tikzpicture}[scale=1.5]
\draw (1,0) -- (1,1) -- (-1,1) -- (-1,-1) -- (0,0) -- cycle;
\end{tikzpicture}
\subcaption{$\phi=3\pi/4$}
\end{minipage}
\begin{minipage}{.30\linewidth}
\begin{tikzpicture}[scale=1.5]
\draw (1,0) -- (1,1) -- (-1,1) -- (-1,-1) -- (0,-1) -- (0,0) -- cycle;
\end{tikzpicture}
\subcaption{$\phi=\pi/2$}
\end{minipage}
\begin{minipage}{.30\linewidth}
\begin{tikzpicture}[scale=1.5]
\draw (1,0) -- (1,1) -- (-1,1) -- (-1,-1) -- (1,-1) -- (1,.-.38268343236508977172) -- (0,0) -- cycle;
\end{tikzpicture}
\subcaption{$\phi=\pi/8$}
\end{minipage}
\caption{L-type domains}
\label{figure_ltype_domains}
\end{figure}
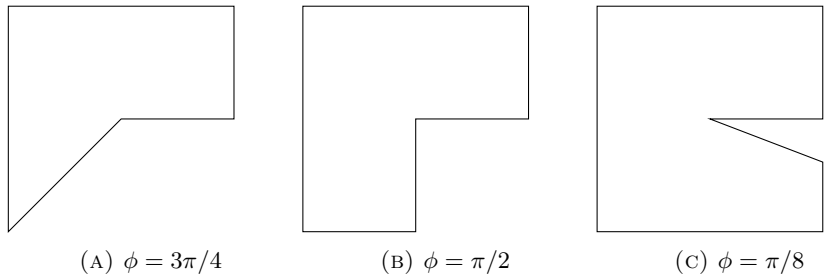

Following \cite{%
chaumontfrelet_ern_vohralik_2021a,%
chaumontfrelet_vohralik_2021a,%
gedicke_geevers_perugia_2019a},
we consider a solution that is singular around the edge $\{(0,0)\} \times (0,1)$.
Specifically, it reads $\BA \eq (0,0,s)$, where
\begin{equation}
s(r,\theta,\bx_3)
\eq
\chi(r)r^\alpha\sin(\alpha\theta)
\end{equation}
with $\alpha \eq \pi/(2\pi-\phi)$ and $\chi$ is a cutoff function employed to satisfy
boundary conditions as in \cite{chaumontfrelet_ern_vohralik_2021a,chaumontfrelet_vohralik_2021a}.

For each angle $\phi$, we perform an adaptive mesh refinement as described above.
In all cases, we start with a mesh generated with $\hmax \eq 2$ and perform three
different refinements for polynomial degrees $p = 0,1$ and $2$.

When $\phi=3\pi/4$, the adaptive process is driven by the cell estimator
$\{\eta_K\}_{K \in \CT_h}$ and the results are reported on Figure
\ref{figure_ltype_small}. As can be seen on the left-panel, the optimal
convergence rates are observed. Notice that for $p=2$, the convergence rate
is not $N_{\rm dofs}^{-(p+1)/3} = N_{\rm dofs}$ because we are using
isotropic elements, see \cite[Section 4.2.3]{apel_1999a} for more details.
Similar to the cube experiment, the effectivity indices of the edge-based
and equilibrated estimators are respectively close to $\sqrt{6}$ and $1$.

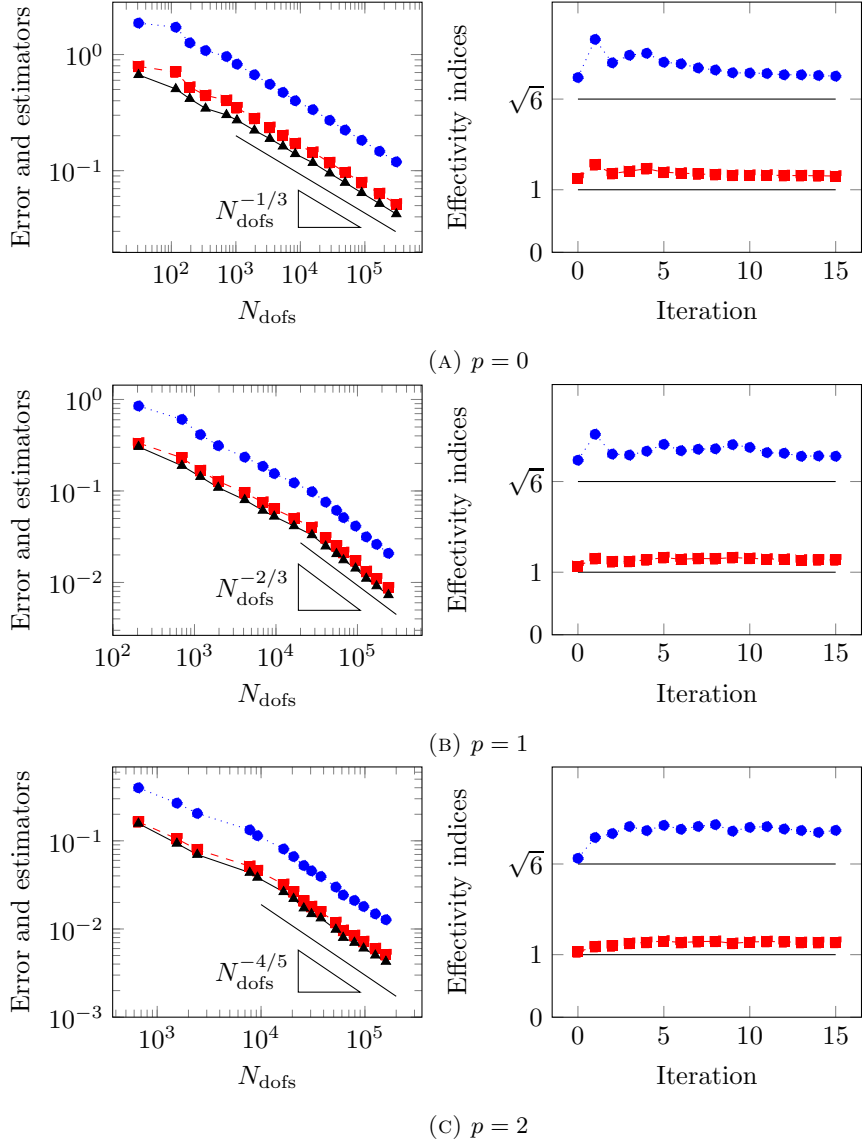
\begin{figure}
\begin{minipage}{\linewidth}
\begin{minipage}{.45\linewidth}
\begin{tikzpicture}
\begin{axis}
[
	width=\linewidth,
	xmode = log,
	ymode = log,
	xlabel = {$N_{\rm dofs}$},
	ylabel = {Error and estimators}
]

\plot[color=blue,mark=*,dotted]
table[x = nr_dofs, y expr = sqrt(6)*\thisrow{este}]
{figures/lbrick_small/data/P0.txt};

\plot[color=red,mark=square*,dashed]
table[x = nr_dofs, y = estc]
{figures/lbrick_small/data/P0.txt};

\plot[color=black,mark=triangle*,solid]
table[x = nr_dofs, y = err]
{figures/lbrick_small/data/P0.txt};

\plot [domain=1e3:3e5] {2*x^(-1./3)};

\SlopeTriangle{.6}{-.2}{.1}{-1./3}{$N_{\rm dofs}^{-1/3}$}{}

\end{axis}
\end{tikzpicture}
\end{minipage}
\begin{minipage}{.45\linewidth}
\begin{tikzpicture}
\begin{axis}
[
	width=\linewidth,
	xlabel={Iteration},
	ylabel={Effectivity indices},
	ymin=0,
	ymax=4,
	ytick={0,1,2.44948974278317809819},
	yticklabels={0,1,$\sqrt{6}$}
]

\plot[color=blue,mark=*,dotted]
table[x = iter, y expr = sqrt(6)*\thisrow{este}/\thisrow{err}]
{figures/lbrick_small/data/P0.txt};

\plot[color=red,mark=square*,dashed]
table[x = iter, y expr = \thisrow{estc}/\thisrow{err}]
{figures/lbrick_small/data/P0.txt};

\plot [domain=0:15] {1};
\plot [domain=0:15] {2.44948974278317809819};

\end{axis}
\end{tikzpicture}
\end{minipage}
\subcaption{$p=0$}
\end{minipage}

\begin{minipage}{\linewidth}
\begin{minipage}{.45\linewidth}
\begin{tikzpicture}
\begin{axis}
[
	width=\linewidth,
	xmode = log,
	ymode = log,
	xlabel = {$N_{\rm dofs}$},
	ylabel = {Error and estimators}
]

\plot[color=blue,mark=*,dotted]
table[x = nr_dofs, y expr = sqrt(6)*\thisrow{este}]
{figures/lbrick_small/data/P1.txt};

\plot[color=red,mark=square*,dashed]
table[x = nr_dofs, y = estc]
{figures/lbrick_small/data/P1.txt};

\plot[color=black,mark=triangle*,solid]
table[x = nr_dofs, y = err]
{figures/lbrick_small/data/P1.txt};

\plot [domain=2e4:3e5] {20*x^(-2./3)};

\SlopeTriangle{.6}{-.2}{.1}{-2./3}{$N_{\rm dofs}^{-2/3}$}{}

\end{axis}
\end{tikzpicture}
\end{minipage}
\begin{minipage}{.45\linewidth}
\begin{tikzpicture}
\begin{axis}
[
	width=\linewidth,
	xlabel={Iteration},
	ylabel={Effectivity indices},
	ymin=0,
	ymax=4,
	ytick={0,1,2.44948974278317809819},
	yticklabels={0,1,$\sqrt{6}$}
]

\plot[color=blue,mark=*,dotted]
table[x = iter, y expr = sqrt(6)*\thisrow{este}/\thisrow{err}]
{figures/lbrick_small/data/P1.txt};

\plot[color=red,mark=square*,dashed]
table[x = iter, y expr = \thisrow{estc}/\thisrow{err}]
{figures/lbrick_small/data/P1.txt};

\plot [domain=0:15] {1};
\plot [domain=0:15] {2.44948974278317809819};

\end{axis}
\end{tikzpicture}
\end{minipage}
\subcaption{$p=1$}
\end{minipage}

\begin{minipage}{\linewidth}
\begin{minipage}{.45\linewidth}
\begin{tikzpicture}
\begin{axis}
[
	width=\linewidth,
	xmode = log,
	ymode = log,
	xlabel = {$N_{\rm dofs}$},
	ylabel = {Error and estimators}
]

\plot[color=blue,mark=*,dotted]
table[x = nr_dofs, y expr = sqrt(6)*\thisrow{este}]
{figures/lbrick_small/data/P2.txt};

\plot[color=red,mark=square*,dashed]
table[x = nr_dofs, y = estc]
{figures/lbrick_small/data/P2.txt};

\plot[color=black,mark=triangle*,solid]
table[x = nr_dofs, y = err]
{figures/lbrick_small/data/P2.txt};

\plot [domain=1e4:2e5] {30*x^(-4./5)};

\SlopeTriangle{.6}{-.2}{.1}{-4./5}{$N_{\rm dofs}^{-4/5}$}{}

\end{axis}
\end{tikzpicture}
\end{minipage}
\begin{minipage}{.45\linewidth}
\begin{tikzpicture}
\begin{axis}
[
	width=\linewidth,
	xlabel={Iteration},
	ylabel={Effectivity indices},
	ymin=0,
	ymax=4,
	ytick={0,1,2.44948974278317809819},
	yticklabels={0,1,$\sqrt{6}$}
]

\plot[color=blue,mark=*,dotted]
table[x = iter, y expr = sqrt(6)*\thisrow{este}/\thisrow{err}]
{figures/lbrick_small/data/P2.txt};

\plot[color=red,mark=square*,dashed]
table[x = iter, y expr = \thisrow{estc}/\thisrow{err}]
{figures/lbrick_small/data/P2.txt};

\plot [domain=0:15] {1};
\plot [domain=0:15] {2.44948974278317809819};

\end{axis}
\end{tikzpicture}
\end{minipage}
\subcaption{$p=2$}
\end{minipage}
\caption{L-type experiment with $\phi=3\pi/4$.}
\label{figure_ltype_small}
\end{figure}

We use the edge estimator $\{\eta_\edge\}_{\edge \in \CE_h}$
to drive the refinements when $\phi = \pi/2$. The results are
similar to the previous case, and are reported on Figure
\ref{figure_ltype_medium}.

\begin{figure}
\begin{minipage}{\linewidth}
\begin{minipage}{.45\linewidth}
\begin{tikzpicture}
\begin{axis}
[
	width=\linewidth,
	xmode = log,
	ymode = log,
	xlabel = {$N_{\rm dofs}$},
	ylabel = {Error and estimators}
]

\plot[color=blue,mark=*,dotted]
table[x = nr_dofs, y expr = sqrt(6)*\thisrow{este}]
{figures/lbrick_medium/data/P0.txt};

\plot[color=red,mark=square*,dashed]
table[x = nr_dofs, y = estc]
{figures/lbrick_medium/data/P0.txt};

\plot[color=black,mark=triangle*,solid]
table[x = nr_dofs, y = err]
{figures/lbrick_medium/data/P0.txt};

\plot [domain=2e3:1e5] {5*x^(-1./3)};

\SlopeTriangle{.6}{-.2}{.1}{-1./3}{$N_{\rm dofs}^{-1/3}$}{}

\end{axis}
\end{tikzpicture}
\end{minipage}
\begin{minipage}{.45\linewidth}
\begin{tikzpicture}
\begin{axis}
[
	width=\linewidth,
	xlabel={Iteration},
	ylabel={Effectivity indices},
	ymin=0,
	ymax=4,
	ytick={0,1,2.44948974278317809819},
	yticklabels={0,1,$\sqrt{6}$}
]

\plot[color=blue,mark=*,dotted]
table[x = iter, y expr = sqrt(6)*\thisrow{este}/\thisrow{err}]
{figures/lbrick_medium/data/P0.txt};

\plot[color=red,mark=square*,dashed]
table[x = iter, y expr = \thisrow{estc}/\thisrow{err}]
{figures/lbrick_medium/data/P0.txt};

\plot [domain=0:15] {1};
\plot [domain=0:15] {2.44948974278317809819};

\end{axis}
\end{tikzpicture}
\end{minipage}
\subcaption{$p=0$}
\end{minipage}

\begin{minipage}{\linewidth}
\begin{minipage}{.45\linewidth}
\begin{tikzpicture}
\begin{axis}
[
	width=\linewidth,
	xmode = log,
	ymode = log,
	xlabel = {$N_{\rm dofs}$},
	ylabel = {Error and estimators}
]

\plot[color=blue,mark=*,dotted]
table[x = nr_dofs, y expr = sqrt(6)*\thisrow{este}]
{figures/lbrick_medium/data/P1.txt};

\plot[color=red,mark=square*,dashed]
table[x = nr_dofs, y = estc]
{figures/lbrick_medium/data/P1.txt};

\plot[color=black,mark=triangle*,solid]
table[x = nr_dofs, y = err]
{figures/lbrick_medium/data/P1.txt};

\plot [domain=1e4:1e5] {80*x^(-2./3)};

\SlopeTriangle{.6}{-.2}{.1}{-2./3}{$N_{\rm dofs}^{-2/3}$}{}

\end{axis}
\end{tikzpicture}
\end{minipage}
\begin{minipage}{.45\linewidth}
\begin{tikzpicture}
\begin{axis}
[
	width=\linewidth,
	xlabel={Iteration},
	ylabel={Effectivity indices},
	ymin=0,
	ymax=4,
	ytick={0,1,2.44948974278317809819},
	yticklabels={0,1,$\sqrt{6}$}
]

\plot[color=blue,mark=*,dotted]
table[x = iter, y expr = sqrt(6)*\thisrow{este}/\thisrow{err}]
{figures/lbrick_medium/data/P1.txt};

\plot[color=red,mark=square*,dashed]
table[x = iter, y expr = \thisrow{estc}/\thisrow{err}]
{figures/lbrick_medium/data/P1.txt};

\plot [domain=0:15] {1};
\plot [domain=0:15] {2.44948974278317809819};

\end{axis}
\end{tikzpicture}
\end{minipage}
\subcaption{$p=1$}
\end{minipage}

\begin{minipage}{\linewidth}
\begin{minipage}{.45\linewidth}
\begin{tikzpicture}
\begin{axis}
[
	width=\linewidth,
	xmode = log,
	ymode = log,
	xlabel = {$N_{\rm dofs}$},
	ylabel = {Error and estimators}
]

\plot[color=blue,mark=*,dotted]
table[x = nr_dofs, y expr = sqrt(6)*\thisrow{este}]
{figures/lbrick_medium/data/P2.txt};

\plot[color=red,mark=square*,dashed]
table[x = nr_dofs, y = estc]
{figures/lbrick_medium/data/P2.txt};

\plot[color=black,mark=triangle*,solid]
table[x = nr_dofs, y = err]
{figures/lbrick_medium/data/P2.txt};

\plot [domain=5e3:1e5] {50*x^(-2./3)};

\SlopeTriangle{.6}{-.2}{.1}{-2./3}{$N_{\rm dofs}^{-2/3}$}{}

\end{axis}
\end{tikzpicture}
\end{minipage}
\begin{minipage}{.45\linewidth}
\begin{tikzpicture}
\begin{axis}
[
	width=\linewidth,
	xlabel={Iteration},
	ylabel={Effectivity indices},
	ymin=0,
	ymax=4,
	ytick={0,1,2.44948974278317809819},
	yticklabels={0,1,$\sqrt{6}$}
]

\plot[color=blue,mark=*,dotted]
table[x = iter, y expr = sqrt(6)*\thisrow{este}/\thisrow{err}]
{figures/lbrick_medium/data/P2.txt};

\plot[color=red,mark=square*,dashed]
table[x = iter, y expr = \thisrow{estc}/\thisrow{err}]
{figures/lbrick_medium/data/P2.txt};

\plot [domain=0:15] {1};
\plot [domain=0:15] {2.44948974278317809819};

\end{axis}
\end{tikzpicture}
\end{minipage}
\subcaption{$p=2$}
\end{minipage}

\caption{L-type experiment with $\phi=\pi/2$.}
\label{figure_ltype_medium}
\end{figure}

Figure \ref{figure_ltype_large} illustrates the last case where $\phi=\pi/8$,
and the adaptive process is driven by the cell estimator. The results are
again similar, except that the use of isotropic elements also reduces the optimal
convergence rate when $p=1$, as expected \cite{apel_1999a}.

\begin{figure}
\begin{minipage}{.45\linewidth}
\begin{tikzpicture}
\begin{axis}
[
	width=\linewidth,
	xmode = log,
	ymode = log,
	xlabel = {$N_{\rm dofs}$},
	ylabel = {Error and estimators}
]

\plot[color=blue,mark=*,dotted]
table[x = nr_dofs, y expr = sqrt(6)*\thisrow{este}]
{figures/lbrick_large/data/P0.txt};

\plot[color=red,mark=square*,dashed]
table[x = nr_dofs, y = estc]
{figures/lbrick_large/data/P0.txt};

\plot[color=black,mark=triangle*,solid]
table[x = nr_dofs, y = err]
{figures/lbrick_large/data/P0.txt};

\plot [domain=5e3:3e5] {4*x^(-1./3)};

\SlopeTriangle{.6}{-.2}{.1}{-1./3}{$N_{\rm dofs}^{-1/3}$}{}

\end{axis}
\end{tikzpicture}
\end{minipage}
\begin{minipage}{.45\linewidth}
\begin{tikzpicture}
\begin{axis}
[
	width=\linewidth,
	xlabel={Iteration},
	ylabel={Effectivity indices},
	ymin=0,
	ymax=4,
	ytick={0,1,2.44948974278317809819},
	yticklabels={0,1,$\sqrt{6}$}
]

\plot[color=blue,mark=*,dotted]
table[x = iter, y expr = sqrt(6)*\thisrow{este}/\thisrow{err}]
{figures/lbrick_large/data/P0.txt};

\plot[color=red,mark=square*,dashed]
table[x = iter, y expr = \thisrow{estc}/\thisrow{err}]
{figures/lbrick_large/data/P0.txt};

\plot [domain=0:15] {1};
\plot [domain=0:15] {2.44948974278317809819};

\end{axis}
\end{tikzpicture}
\end{minipage}

\begin{minipage}{.45\linewidth}
\begin{tikzpicture}
\begin{axis}
[
	width=\linewidth,
	xmode = log,
	ymode = log,
	xlabel = {$N_{\rm dofs}$},
	ylabel = {Error and estimators}
]

\plot[color=blue,mark=*,dotted]
table[x = nr_dofs, y expr = sqrt(6)*\thisrow{este}]
{figures/lbrick_large/data/P1.txt};

\plot[color=red,mark=square*,dashed]
table[x = nr_dofs, y = estc]
{figures/lbrick_large/data/P1.txt};

\plot[color=black,mark=triangle*,solid]
table[x = nr_dofs, y = err]
{figures/lbrick_large/data/P1.txt};

\plot [domain=1e4:2e5] {15*x^(-8./15)};

\SlopeTriangle{.6}{-.2}{.1}{-8./15}{$N_{\rm dofs}^{-8/15}$}{}

\end{axis}
\end{tikzpicture}
\end{minipage}
\begin{minipage}{.45\linewidth}
\begin{tikzpicture}
\begin{axis}
[
	width=\linewidth,
	xlabel={Iteration},
	ylabel={Effectivity indices},
	ymin=0,
	ymax=4,
	ytick={0,1,2.44948974278317809819},
	yticklabels={0,1,$\sqrt{6}$}
]

\plot[color=blue,mark=*,dotted]
table[x = iter, y expr = sqrt(6)*\thisrow{este}/\thisrow{err}]
{figures/lbrick_large/data/P1.txt};

\plot[color=red,mark=square*,dashed]
table[x = iter, y expr = \thisrow{estc}/\thisrow{err}]
{figures/lbrick_large/data/P1.txt};

\plot [domain=0:15] {1};
\plot [domain=0:15] {2.44948974278317809819};

\end{axis}
\end{tikzpicture}
\end{minipage}

\begin{minipage}{.45\linewidth}
\begin{tikzpicture}
\begin{axis}
[
	width=\linewidth,
	xmode = log,
	ymode = log,
	xlabel = {$N_{\rm dofs}$},
	ylabel = {Error and estimators}
]

\plot[color=blue,mark=*,dotted]
table[x = nr_dofs, y expr = sqrt(6)*\thisrow{este}]
{figures/lbrick_large/data/P2.txt};

\plot[color=red,mark=square*,dashed]
table[x = nr_dofs, y = estc]
{figures/lbrick_large/data/P2.txt};

\plot[color=black,mark=triangle*,solid]
table[x = nr_dofs, y = err]
{figures/lbrick_large/data/P2.txt};

\plot [domain=5e4:3e5] {15*x^(-8./15)};

\SlopeTriangle{.6}{-.2}{.1}{-8./15}{$N_{\rm dofs}^{-8/15}$}{}

\end{axis}
\end{tikzpicture}
\end{minipage}
\begin{minipage}{.45\linewidth}
\begin{tikzpicture}
\begin{axis}
[
	width=\linewidth,
	xlabel={Iteration},
	ylabel={Effectivity indices},
	ymin=0,
	ymax=4,
	ytick={0,1,2.44948974278317809819},
	yticklabels={0,1,$\sqrt{6}$}
]

\plot[color=blue,mark=*,dotted]
table[x = iter, y expr = sqrt(6)*\thisrow{este}/\thisrow{err}]
{figures/lbrick_large/data/P2.txt};

\plot[color=red,mark=square*,dashed]
table[x = iter, y expr = \thisrow{estc}/\thisrow{err}]
{figures/lbrick_large/data/P2.txt};

\plot [domain=0:15] {1};
\plot [domain=0:15] {2.44948974278317809819};

\end{axis}
\end{tikzpicture}
\end{minipage}

\caption{L-type experiment with $\phi=\pi/8$.}
\label{figure_ltype_large}
\end{figure}

Interestingly, we observe that the effectivity indices $\effe$ and $\effc$
respectively stays above $\sqrt 6$ and $1$ (apart from the first iterations
when $\phi = \pi/2$, which is due to data oscillation). Notice that our theory does
not cover this estimate, as it requires the constant $\Clift$, which is not
$1$ in this case. The fact that the effictivity indices remains independent
of the angle $\phi$ hints that the $\Clift$ may not be compulsory to obtain guaranteed
estimates.

\subsection{Corner singularity in the Fichera cube}

In this last experiment, we consider the Fichera domain
$\Omega \eq (-1,1)^3 \setminus (0,1)^3$. We select the
right-hand side $\BJ \eq (1,1,0)$. Starting from an initial
mesh generated with $\hmax \eq 2$, we consider to adaptive
refinement processes. We first set $p=0$ and employ the
edge-based estimator to drive the adaptive process and then,
we use the cell-based estimator with $p=1$.

Figure \ref{figure_fichera} presents the results. As for the 
L-type domain example, we obtain the optimal convergence rates,
and the estimators seem to provide guaranteed upper bounds even
when $\Clift$ is omitted.

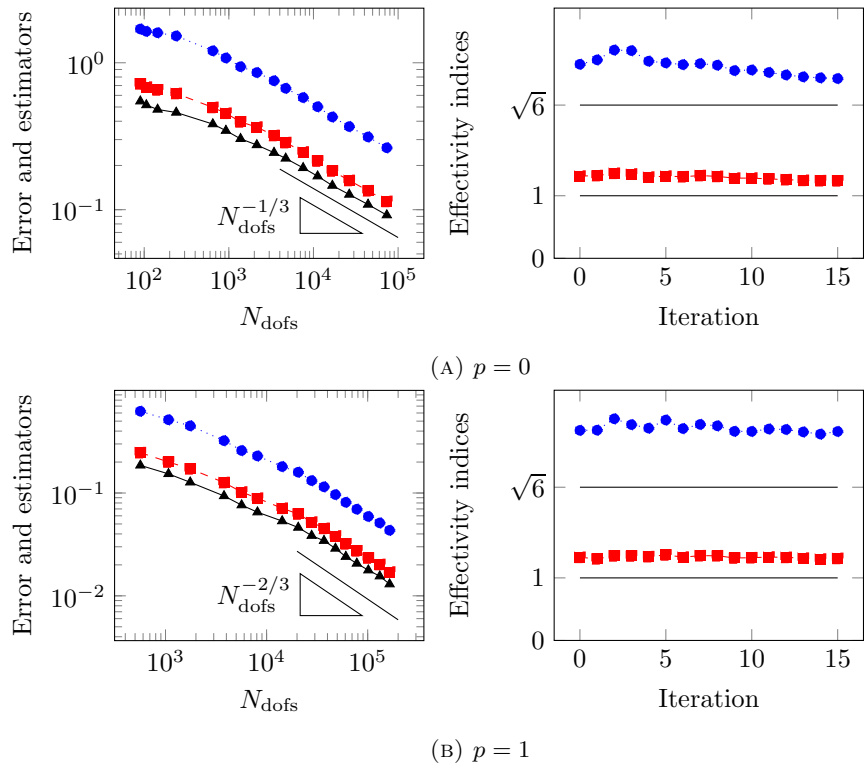
\begin{figure}
\begin{minipage}{\linewidth}
\begin{minipage}{.45\linewidth}
\begin{tikzpicture}
\begin{axis}
[
	width=\linewidth,
	xmode = log,
	ymode = log,
	xlabel = {$N_{\rm dofs}$},
	ylabel = {Error and estimators}
]

\plot[color=blue,mark=*,dotted]
table[x = nr_dofs, y expr = sqrt(6)*\thisrow{este}]
{figures/fichera/data/P0.txt};

\plot[color=red,mark=square*,dashed]
table[x = nr_dofs, y = estc]
{figures/fichera/data/P0.txt};

\plot[color=black,mark=triangle*,solid]
table[x = nr_dofs, y = err]
{figures/fichera/data/P0.txt};

\plot [domain=4e3:1e5] {3*x^(-1./3)};

\SlopeTriangle{.6}{-.2}{.1}{-1./3}{$N_{\rm dofs}^{-1/3}$}{}

\end{axis}
\end{tikzpicture}
\end{minipage}
\begin{minipage}{.45\linewidth}
\begin{tikzpicture}
\begin{axis}
[
	width=\linewidth,
	xlabel={Iteration},
	ylabel={Effectivity indices},
	ymin=0,
	ymax=4,
	ytick={0,1,2.44948974278317809819},
	yticklabels={0,1,$\sqrt{6}$}
]

\plot[color=blue,mark=*,dotted]
table[x = iter, y expr = sqrt(6)*\thisrow{este}/\thisrow{err}]
{figures/fichera/data/P0.txt};

\plot[color=red,mark=square*,dashed]
table[x = iter, y expr = \thisrow{estc}/\thisrow{err}]
{figures/fichera/data/P0.txt};

\plot [domain=0:15] {1};
\plot [domain=0:15] {2.44948974278317809819};

\end{axis}
\end{tikzpicture}
\end{minipage}
\subcaption{$p=0$}
\end{minipage}

\begin{minipage}{\linewidth}
\begin{minipage}{.45\linewidth}
\begin{tikzpicture}
\begin{axis}
[
	width=\linewidth,
	xmode = log,
	ymode = log,
	xlabel = {$N_{\rm dofs}$},
	ylabel = {Error and estimators}
]

\plot[color=blue,mark=*,dotted]
table[x = nr_dofs, y expr = sqrt(6)*\thisrow{este}]
{figures/fichera/data/P1.txt};

\plot[color=red,mark=square*,dashed]
table[x = nr_dofs, y = estc]
{figures/fichera/data/P1.txt};

\plot[color=black,mark=triangle*,solid]
table[x = nr_dofs, y = err]
{figures/fichera/data/P1.txt};

% \plot [domain=2e4:2e5] {5*(x/ln(x))^(-2./3)};
\plot [domain=2e4:2e5] {20*x^(-2./3)};

\SlopeTriangle{.6}{-.2}{.1}{-2./3}{$N_{\rm dofs}^{-2/3}$}{}

\end{axis}
\end{tikzpicture}
\end{minipage}
\begin{minipage}{.45\linewidth}
\begin{tikzpicture}
\begin{axis}
[
	width=\linewidth,
	xlabel={Iteration},
	ylabel={Effectivity indices},
	ymin=0,
	ymax=4,
	ytick={0,1,2.44948974278317809819},
	yticklabels={0,1,$\sqrt{6}$}
]

\plot[color=blue,mark=*,dotted]
table[x = iter, y expr = sqrt(6)*\thisrow{este}/\thisrow{err}]
{figures/fichera/data/P1.txt};

\plot[color=red,mark=square*,dashed]
table[x = iter, y expr = \thisrow{estc}/\thisrow{err}]
{figures/fichera/data/P1.txt};

\plot [domain=0:15] {1};
\plot [domain=0:15] {2.44948974278317809819};

\end{axis}
\end{tikzpicture}
\end{minipage}
\subcaption{$p=1$}
\end{minipage}

\caption{Fichera domain example}
\label{figure_fichera}
\end{figure}

\section{Conclusion}
\label{section_conclusion}

We propose two a posteriori error estimators that are motivated by
a novel Prager--Synge identity for the curl--curl problem. Both estimators
are polynomial-degree-robust, and rely on divergence-constrained minimization
problems over edge patches. When the domain is convex, these estimators also
provide guaranteed and fully computable upper bounds. In the general case
however, the reliability estimate involves a constant $\Clift$ related to
regular decompositions of $\BH(\ccurl)$ fields that is not easily computable in practice.

We present a set of numerical examples which illustrates the key theoretical
findings: both estimators are reliable and efficient with $p$-robust constants,
and provide guaranteed upper bounds in convex domains. In addition, when the domain
is not convex, these numerical tests suggest that the estimators can still be employed
without the constant $\Clift$ to provide a guaranteed upper bound. While we are not
able to prove this result, we provide some theoretical reasons why it may be the case.
We also employ both estimators to drive adaptive mesh refinements, and obtain optimal
convergence rates in domains featuring re-entrant edges and corners.

In practice, it seems that the cell-based equilibrated estimator should be preferred,
as it provides improved results as compared to the edge-based estimator, only at the
price of a moderate additional complexity in the implementation. From a theoretical viewpoint
however, the edge-based estimator may be preferable since the associated efficiency estimates
involve smaller mesh patches, which may be of interest to design adaptive refinement algorithms
that provably converge with optimal rates.

\bibliographystyle{amsplain}
\bibliography{bibliography}

\end{document}